\documentclass[11pt,letterpaper,center]{amsart}

\usepackage{xcolor}
\usepackage{amsmath,amssymb}
\usepackage{enumitem}
\setenumerate{label=(\arabic*)}
\usepackage{tikz}
\usepackage{ytableau}
\ytableausetup{boxsize=0.3em}
\usetikzlibrary{cd,patterns}

\usepackage[bb=ams,scr=euler]{mathalpha}

\usepackage{hyperref}

\usepackage{thmtools}
\declaretheorem[name=Theorem]{theorem}
\declaretheorem[name=Question,sibling=theorem]{question}
\declaretheorem[name=Conjecture,sibling=theorem]{conjecture}
\declaretheorem[name=Corollary,sibling=theorem]{corollary}
\declaretheorem[name=Claim,sibling=theorem,numbered=no]{claim}
\declaretheorem[name=Proposition,sibling=theorem]{prop}
\declaretheorem[name=Lemma,sibling=theorem]{lemma}
\declaretheorem[name=Definition,sibling=theorem]{definition}
\declaretheorem[name=Remark,sibling=theorem]{remark}
\makeatletter
\newcommand{\abs}[1]{|#1|}
\newcommand{\bd}{\partial}
\newcommand{\C}{\mathbf{C}}
\renewcommand{\d}{\mathrm{d}}

\newcommand{\pd}[2]{\frac{\partial #1}{\partial #2}}
\newcommand{\Q}{\mathbf{Q}}
\newcommand{\R}{\mathbf{R}}
\newcommand{\Z}{\mathbf{Z}}

\def\@secnumfont{\bfseries}
\renewcommand\section{\@startsection{section}{1}{0pt}{-3.5ex \@plus -1ex \@minus -.2ex}{2.3ex \@plus.2ex}{\centering\itshape}}

\newcommand{\set}[1]{\{#1\}}

\renewcommand{\subsection}{\@startsection{subsection}{2}%
  \z@{.5\linespacing\@plus.7\linespacing}{-.5em}%
  {\normalfont\itshape}}

\renewcommand{\paragraph}{\@startsection{paragraph}{4}%
  \z@{-.3em}\z@{\normalfont\itshape}}

\def\l@paragraph{\@tocline{4}{0pt}{1pc}{7pc}{}}
\makeatother

\title{On the spectral diameter of the Grassmannians}
\author[H. Alizadeh]{Habib Alizadeh}
\email{habib.alizadeh.math@gmail.com}
\address{The Institute of Geometry and Physics, University of Science and Technology of China, 96 Jinzhai Road, Baohe District, Hefei, Anhui, 230000, China}

\author[M. S. Atallah]{Marcelo S. Atallah}
\email{atallah.marcelo@gmail.com}
\address{School of Mathematical and Physical Sciences, University of Sheffield, Hicks Building, Sheffield, S10 2TN, England}

\author[D. Cant]{Dylan Cant}
\email{dylan@dylancant.ca}
\address{Départment de mathématiques et de statistique, Université de Montréal, Pavillon André-Aisenstadt, 2920 Chemin de la Tour, Montréal, Québec, H3T 1J4, Canada}

\author[J. Shang]{Jianqiao Shang}
\email{jianqiao.shang@universite-paris-saclay.fr}
\address{Institut de mathématique d'Orsay, Université Paris-Saclay, Bâtiment 307, rue Michel Magat, F-91405 Orsay Cedex, France}

\begin{document}
\begin{abstract}
  The diameter of the spectral pseudometric on the universal cover of the Hamiltonian diffeomorphism group of $\mathrm{Gr}(2,p)$ is shown to be finite whenever $p$ is a prime number. On the other hand, it is shown that the diameter is infinite in the case of $\mathrm{Gr}(2k,2n)$ for all natural numbers $k<n$.
\end{abstract}
\maketitle


\section{Introduction}
\label{sec:introduction}

For any compact symplectic manifold\footnote{The pseudometric depends on the well-definedness of Hamiltonian Floer homology, which generally requires some sort of hypotheses on $(M,\omega)$; see \cite{hofer-salamon-95,mcduff-salamon-book-2012}. In this paper we consider only monotone symplectic manifolds.} $(M,\omega)$ one can construct a pseudometric $\gamma$ on the universal cover of the Hamiltonian diffeomorphism group $\mathrm{Ham}(M)$ using Floer theory \cite{floer-CMP-1989}. The pseudometric can be pushed down to a genuine metric on $\mathrm{Ham}(M)$.\footnote{The pseudometric $\gamma$ on the universal cover may be degenerate; see \cite{humilière-jannaud-leclercq-arXiv-2023}.} The construction dates back to \cite{viterbo92-GF,schwarz_spectral_invariants,oh-2005-duke} and depends on a choice of coefficient ring (typically a field $\mathbf{F}$ but other rings such as $\mathbf{Z}$ are also considered; see \cite{kawamoto-shelukhin-AIM-2024}). This pseudometric is the subject of active research: \cite{kislev-shelukhin-GT-2021,kawamoto-continuity-2022,kawamoto-shelukhin-AIM-2024,alizadeh-atallah-cant-arxiv-2024,sun-arXiv-2024,alizadeh-atallah-cant-math-z-2025}. To motivate the question we study, we recall the well-known folklore conjecture (see \cite[pp.\,300]{buhovsky-humiliere-seyfaddini-2021}) on the diameter of $\gamma$:
\begin{conjecture}\label{conjecture:1}
  If $\omega$ vanishes on $\pi_{2}(M)$ then the diameter of $\gamma$ is infinite.
\end{conjecture}
The conjecture has been established in certain cases, see, e.g., \cite{usher-ASENS-2013,kislev-shelukhin-GT-2021} and \cite{ganor-tanny-AGT-2023,mailhot-spectral-diameter-2024}.\footnote{The criterion in \cite{usher-ASENS-2013} is the existence of a non-constant autonomous Hamiltonian function $H$ which generates a system whose contractible orbits, of any non-zero period, are constant; see also \cite{kislev-shelukhin-GT-2021}. The criterion in \cite{mailhot-spectral-diameter-2024} is the existence of an incompressible embedding of a Liouville domain $\Omega$ with $\mathrm{SH}(\Omega)\ne 0$, in addition to the hypothesis that $\omega$ vanishes on $\pi_{2}(M)$; here $\mathrm{SH}(\Omega)$ is the \emph{symplectic cohomology} of $\Omega$; see, e.g., \cite{seidel-IP-2008}.}

In contrast to Conjecture \ref{conjecture:1}, if there \emph{are} symplectic spheres in $M$, then upper bounds on the spectral diameter\footnote{We adopt the following terminology: the \emph{spectral diameter} of $M$ is the diameter of the pseudometric $\gamma$ on the universal cover of the Hamiltonian diffeomorphism group.} can sometimes be established, see \cite{entov-poltero-IMRN-2003}; the only known method uses the \emph{quantum cohomology ring}; see Lemma \ref{lemma:folklore}.

In this paper, we focus on the case when $M=\mathrm{Gr}(k,n)$ is the Grassmannian of complex $k$-planes in $\C^{n}$. The quantum cohomology is famously studied in \cite{witten-arXiv-1993}; see also \cite{bertram-aim-1997,siebert-tian-AJM-1997,rietsch-duke-2001,buch-compositio-2003,galkin-golyshev-RMS-2006,castronovo-QT-2023}.

Our first result concerns $k=2$:

\begin{theorem}\label{theorem:main-over-Q}
  For $\mathrm{Gr}(2,n)$, the spectral diameter, over a coeffiecient field of characteristic zero, is finite if $n$ is a prime number.
\end{theorem}
In fact, we prove a more general statement:
\begin{theorem}\label{theorem:main}
  For $\mathrm{Gr}(2,n)$, the spectral diameter, over a coefficient field $\mathbf{F}$ of nonzero characteristic $p$, is finite if $n\ne p$ is a prime number and the subset $\set{p,-1}$ generates the group of units $(\Z/n\Z)^{\times}$.
\end{theorem}
One recovers Theorem \ref{theorem:main-over-Q} from \ref{theorem:main} from a simple ``specialization to primes'' number theory argument; see \S\ref{sec:proof-theorem-main-over-Q}. Besides the intrinsic interest of understanding the geometry of the spectral pseudometric, bounds on the diameter $\gamma$ force Lagrangian intersections.\footnote{A Lagrangian submanifold $L\subset (M,\omega)$ satisfies $2\dim L=\dim M$ and $\omega|_{TL}=0$.} In fact, the quantum cohomology of the ambient space strongly influences ``enumerative invariants'' for Lagrangian submanifolds; see, in particular, \cite{auroux-JGGT-2007,biran-cornea-rigidity-uniruling-2009,abouzaid-IHES-2010,biran-cornea-GT-2012,seidel-invent-2014,sheridan-pubihes-2016,FOOO-MEMO-2019,castronovo-thesis-2021,castronovo-QT-2023}; we return to this subject in \S\ref{sec:lagr-inters}. Let us merely comment now that using Lagrangian submanifolds we are able to show:
\begin{theorem}\label{theorem:2k-2n}
  The spectral diameter of $\mathrm{Gr}(2k,2n)$ is infinite using a field of characteristic zero, provided that $k<n$.
\end{theorem}
Theorems \ref{theorem:main} and \ref{theorem:2k-2n} should be considered the main theorems of this paper. Let us introduce the quantum cohomology: in this paper, we only consider $M$ which are (positively) monotone.\footnote{In this paper, positively monotone means the following: there exists some smooth 2-sphere $u:S^{2}\to M$ with positive symplectic area, and $c_{1}(u)=N\omega(u)$ for all 2-spheres; we adopt the convention that $\omega$ is normalized so that $\omega(\pi_{2}(M))=\Z$, so that the number $N$ is the so-called \emph{minimal Chern number} of $M$. Let us comment here that \cite{kawamoto-continuity-2022,sun-arXiv-2024} explore the spectral diameter of \emph{negatively} monotone manifolds.} The quantum cohomology is:
\begin{equation}\label{eq:quantum-cohomology}
  \mathrm{QH}^{*}(M;\mathbf{F}):=\mathrm{H}^{*}(M;\mathbf{F})\otimes \mathbf{F}[q^{-1},q];
\end{equation}
where $q$ is a formal variable; the multiplication is the quantum cup product; see \S\ref{sec:background} for further discussion. The only known method to bound the spectral diameter from above is the following result from \cite{entov-poltero-IMRN-2003}:
\begin{lemma}[Entov--Polterovich]\label{lemma:folklore}
  If each nonzero homogeneous element in $\mathrm{QH}^{*}(M;\mathbf{F})$ is invertible, then the spectral diameter of $M$ computed using the field $\mathbf{F}$ is finite.
\end{lemma}
The proof is recalled in \S\ref{sec:proof-lemma-folklore}.
Theorem \ref{theorem:main} is proved using Lemma \ref{lemma:folklore}. While the hypotheses of Lemma \ref{lemma:folklore} are not preserved under field extensions, its conclusion does bound the diameter of $\gamma$ over other fields:\footnote{The spectral invariants are generally stable under various types of extensions of coefficients see, e.g., \cite[\S7]{usher-ASENS-2013} and \cite[\S2.5]{usher-zhang-GT-2016}; see also \cite[Proposition 3.1.5]{kawamoto-shelukhin-AIM-2024} for a general statement.}
\begin{lemma}\label{lemma:coefficient-extension}
  If the spectral diameter over $\mathbf{F}$ is finite, and there is a field extension $\mathbf{F}\to \mathbf{K}$, then the spectral diameter over $\mathbf{K}$ is finite.
\end{lemma}

Our next result considers the non-applicability of Lemma \ref{lemma:folklore} to the other Grassmannians; we restrict to the range $n\ge 2k$ to avoid the isomorphism $\mathrm{Gr}(k,n)\to \mathrm{Gr}(n-k,n)$:
\begin{theorem}\label{theorem:only-applies}
  For $\mathrm{Gr}(k,n)$, $n\ge 2k$, and $\mathbf{F}$ a field of characteristic $p$, the hypothesis of Lemma \ref{lemma:folklore} is satisfied if and only if the degree zero subring $\mathrm{QH}^{0}(M;\mathbf{F})$ is a field. This holds if and only if:
  \begin{itemize}
  \item $k=1$, or
  \item $k=2$, $n$ is prime, $p\ne n$, and, if $p\ne 0$, $\set{p,-1}$ generates $(\Z/n\Z)^{\times}$.
  \end{itemize}
  In all other cases, Lemma \ref{lemma:folklore} does not apply.
\end{theorem}
Because Lemma \ref{lemma:folklore} is the only known way to bound the spectral diameter, it is natural to wonder whether the spectral diameter of $\mathrm{Gr}(k,n)$ is infinite whenever Theorems \ref{theorem:main-over-Q} or \ref{theorem:main} do not apply; we discuss this further in \S\ref{sec:further-questions}.

\subsection{Background}
\label{sec:background}

Before we sketch the proofs, we recall some necessary background about the quantum cohomology of Grassmannians. The ordinary cohomology $H^{*}(\mathrm{Gr}(k,n);\Z)$ is the free graded $\Z$-module generated by variables $\sigma_{D}$ where $D$ is a Young diagram fitting inside $[0,n-k]\times [-k,0]$.

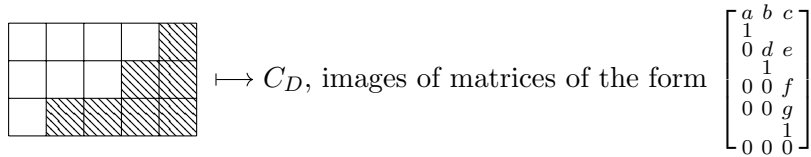
\begin{figure}[h]
  \centering
  \begin{tikzpicture}[scale=.5]
    \fill[pattern=north west lines] (1,0)--(5,0)--(5,3)--(4,3)--(4,2)--(3,2)--(3,1)--(1,1)--cycle;
    \draw (0,0) grid (5,3);
    \begin{scope}[shift={(5.5,0)}]
      \node (A) at (1,1.5) [right] {$C_{D},$ images of matrices of the form $
        \left[\begin{smallmatrix}
          a&b&c\\
          1&&\\
          0&d&e\\
          &1&\\
          0&0&f\\
          0&0&g\\
          &&1\\
          0&0&0
        \end{smallmatrix}\right]
      $};
    \draw[|->] (0,1.5)--(A);
    \end{scope}
  \end{tikzpicture}
  \caption{A Young diagram $d\subset [0,n-k]\times [0,k]$ determines a Schubert cycle $C_{D}\subset \mathrm{Gr}(k,n)$. The diagram consists of the unshaded squares. Each vertical edge corresponds to a special pivot row of the matrix; other rows have a certain number of free variables. The class $\sigma_{D}$ is represented by $C_{D}$.}
  \label{fig:young-diagram}
\end{figure}

The tensor product $H^{*}(\mathrm{Gr}(k,n);\Z)\otimes \Z[q]$ --- we add a formal ``quantum'' variable $q$ --- has the structure of a graded unital algebra over $\Z[q]$; the product is the so-called \emph{quantum cup product}; see \eqref{eq:quantum-intersection-product} below. The (complex) degree of $\sigma_{D}$ is the number of squares\footnote{From the perspective of real dimensions, the grading is twice the number of squares. Throughout this paper, we adopt the following convention: whenever we speak about the quantum cohomology of complex Grassmannians, we will use complex degrees, and whenever we speak about the quantum cohomology of a ``general'' symplectic manifold we will use real degrees.} in $D$, and $q$ is a formal variable of (complex) degree\footnote{The number $n$ is the minimal Chern number of $\mathrm{Gr}(k,n)$.} $n$.

The \emph{quantum cup product} is then defined as a deformation of the classical cup product:
\begin{equation}\label{eq:quantum-intersection-product}
  \sigma_{D_{1}}\ast \sigma_{D_{2}}=\sigma_{D_{1}}\smile \sigma_{D_{2}}+\sum_{k=1}^{\infty}q^{k}\sigma_{D_{1}}\ast_{k}\sigma_{D_{2}},
\end{equation}
where $\ast_{k}$ is a commutative product of degree $-nk$ which counts holomorphic spheres $u$ with area $\omega(u)=k$ (the coefficients defining this operation are given by certain three pointed Gromov-Witten invariants); the details are standard by now, and refer to \cite{mcduff-salamon-book-2012} for generalities. It is known that total operation $\ast$ is graded, associative, and $1=\sigma_{\emptyset}$ is the unit element.

As is common when using the quantum cohomology in the study Hamiltonian Floer theory on monotone symplectic manifolds, we consider a localization of this algebra:
\begin{definition}
  For any field $\mathbf{F}$, the quantum cohomology algebra over $\mathbf{F}$ is defined by:
  \begin{equation*}
    \mathrm{QH}^{*}(\mathrm{Gr}(k,n);\mathbf{F}):=(H^{*}(\mathrm{Gr}(k,n);\Z)\otimes \Z[q])\otimes_{\Z[q]} \mathbf{F}[q^{-1},q];
  \end{equation*}
  the product structure is the one obtained by localizing \eqref{eq:quantum-intersection-product}.
\end{definition}

\subsection{Quantum Pieri rule}
\label{sec:quantum-pieri-rule}

In the case of Grassmannians, the relevant three-pointed Gromov-Witten invariants appearing in \eqref{eq:quantum-intersection-product} have been calculated, see \cite{bertram-aim-1997,buch-compositio-2003}. The product is completely determined by the so-called \emph{quantum Pieri rule} (the sums are described below):
\begin{equation}\label{eq:quantum-pieri-rule}
  x_{j}\ast \sigma_{D}=\sum_{D'}\sigma_{D'}+\sum_{D''}q\sigma_{D''}.
\end{equation}
Here $x_{j}$ is the $j$th Chern class of the obvious $k$-plane bundle over $\mathrm{Gr}(k,n)$; it is known that $x_{j}$ is the basis element corresponding to the diagram with $j$ boxes in the first column. The first sum (the classical contribution) is over all Young diagrams $D'$ which are obtained from $D$ by adding $j$ squares in such a way that:
\begin{itemize}
\item each added square is either on the left of the rectangle, or has a square in $D$ to the left of it (we do not allow an added square to have another added square to its left).
\end{itemize}
The second sum (the quantum contribution) is over all diagrams $D''$ obtained by removing $n-j$ squares from $D$ in such a way that:
\begin{itemize}
\item all possible squares in the top row must be removed --- \emph{this requires $D$ to have all of the squares in the top row!}
\item each additional square which is removed (i.e., not on the top row) must either be on the left of the rectangle, or must have the square to its left unremoved.
\item After these removals, one should shift the result up by one unit to obtain $D''$.
\end{itemize}
We provide an example computation in $\mathrm{Gr}(3,6)$:
\begin{equation*}
  \sigma_{\ydiagram{1,1}}\ast \sigma_{\ydiagram{3,1}}=\sigma_{\ydiagram{3,2,1}}+q\sigma_{\emptyset}.
\end{equation*}
In particular, it holds that $x_{k}^{n}=q^{k}$; later on we use this to show:
\begin{lemma}\label{lemma:when-is-field}
  The quantum cohomology $\mathrm{QH}^{*}(\mathrm{Gr}(k,n),\mathbf{F})$ satisfies the hypothesis of Lemma \ref{lemma:folklore} if and only if the degree zero part $\mathrm{QH}^{0}(\mathrm{Gr}(k,n),\mathbf{F})$ is a field; in particular, the first part of Theorem \ref{theorem:only-applies} holds.
\end{lemma}

\begin{remark}
  This lemma may perhaps seem to be some sort of algebraic triviality, but we should mention that the same conclusion does not hold over symplectically aspherical manifolds if one defines $\mathrm{QH}^{*}(M,\mathbf{F})=H^{*}(M,\mathbf{F})$ (as is the standard convention when working with spectral invariants). The proof of Lemma \ref{lemma:when-is-field} uses the fact that $\mathrm{QH}^{\dim M}(M,\mathbf{F})$ contains an invertible element. For Grassmannians, the above shows we can take $q^{-k}x_{k}^{n}$.
\end{remark}

\subsection{The case of $\mathrm{Gr}(2,n)$ where $n$ is odd.}
\label{sec:sketch-proof-theorem-main}

The following technical result is the key ingredient in our proofs of Theorems \ref{theorem:main} and \ref{theorem:main-over-Q}:
\begin{theorem}\label{theorem:technical-main}
  Suppose $n=2\ell+1$. The degree $2\ell-1$ part: $$\mathrm{QH}^{2\ell-1}(\mathrm{Gr}(2,n),\mathbf{F})$$ is $\ell$ dimensional and generated by $\sigma_{D}$ where $D$ has $j=0,\dots,\ell-1$ boxes in the second row; see \eqref{eq:case-of-ell-5} for the standard basis in the case $\ell=6$, $n=13$.
  \begin{equation}
    \label{eq:case-of-ell-5}    \left\{\hspace{.1cm}\ydiagram{11}\hspace{.1cm},\hspace{.4cm}\ydiagram{10,1}\hspace{.1cm},\hspace{.4cm}\ydiagram{9,2}\hspace{.1cm},\hspace{.4cm}\ydiagram{8,3}\hspace{.1cm},\hspace{.4cm}\ydiagram{7,4}\hspace{.1cm},\hspace{.4cm}\ydiagram{6,5}\hspace{.1cm}\right\}.
  \end{equation}
  With respect to this basis, multiplication by a well-chosen degree zero element $A$ has matrix:
  \begin{equation*}
    M=\left[
      \begin{smallmatrix}
        +1&-1&0&0&0&0\\
        -1&0&-1&0&0&0\\
        0&-1&0&-1&0&0\\
        0&0&-1&0&-1&0\\
        0&0&0&-1&0&-1\\
        0&0&0&0&-1&0\\
      \end{smallmatrix}
    \right]\text{ (shown here with $n=13$).}
  \end{equation*}
  If $\pi$ is the characteristic polynomial of this matrix, then:
  \begin{equation*}
    x^{\ell}\pi(-x-x^{-1})=x^{n-1}+x^{n-2}+\dots+x+1.
  \end{equation*}
  Consequently, if $\mathbf{F}=\Z/p\Z$, the following are equivalent:
  \begin{enumerate}
  \item\label{item:technical-main-1} $\mathrm{QH}^{0}(\mathrm{Gr}(2,n),\mathbf{F})$ is a field,
  \item\label{item:technical-main-2} $\pi(y)$ is irreducible over $\mathbf{F}[y]$,
  \item\label{item:technical-main-3} $n$ is a prime number, coprime with $p=\mathrm{char}(\mathbf{F})$, and $\set{-1,p}$ generate the group of units $(\Z/n\Z)^{\times}$ .
  \end{enumerate}
\end{theorem}
The proof of Theorem \ref{theorem:technical-main} is given in \S\ref{sec:proof-theorem-tech-main}.

\subsection{The case of $\mathrm{Gr}(2,n)$ for even $n$}
\label{sec:the case of-even-n}

For even numbers $n\ge 4$, the degree zero part of quantum cohomology is never a field (for any choice of coefficients). However, as long as $n$ is coprime to $\mathrm{char}(\mathbf{F})$, the algebra is still \emph{semisimple} in that $\mathrm{QH}^{0}(\mathrm{Gr}(2,n);\mathbf{F})$ splits into a direct sum of fields; see Proposition \ref{proposition:secret-semisimple}. To establish such a result, we use an analogue of Theorem~\ref{theorem:technical-main} for the case $n=2\ell+2$.

\begin{theorem}\label{theorem:matrix-when-n-even}
  Suppose $n=2\ell+2$. The degree $2\ell$ part $\mathrm{QH}^{2\ell}(\mathrm{Gr}(2,n),\mathbf{F})$ is $\ell+1$ dimensional and generated by $\sigma_{D}$ where $D$ has $j=0,\dots,\ell$ boxes in the second row; see \eqref{eq:case-of-ell-5-even} for the standard basis in the case $\ell=5$, $n=12$.
  \begin{equation}
    \label{eq:case-of-ell-5-even}    \left\{\hspace{.1cm}\ydiagram{10}\hspace{.1cm},\hspace{.4cm}\ydiagram{9,1}\hspace{.1cm},\hspace{.4cm}\ydiagram{8,2}\hspace{.1cm},\hspace{.4cm}\ydiagram{7,3}\hspace{.1cm},\hspace{.4cm}\ydiagram{6,4}\hspace{.1cm},\hspace{.4cm}\ydiagram{5,5}\hspace{.1cm}\right\}.
  \end{equation}
  With respect to this basis, multiplication by a well-chosen degree zero element $A$ has matrix:
  \begin{equation*}
    M=\left[
      \begin{smallmatrix}
        +1&-1&0&0&0&0\\
        -1&0&-1&0&0&0\\
        0&-1&0&-1&0&0\\
        0&0&-1&0&-1&0\\
        0&0&0&-1&0&-1\\
        0&0&0&0&-1&+1\\
      \end{smallmatrix}
    \right]\text{ (shown here with $n=12$).}
  \end{equation*}
  If $\pi$ is the characteristic polynomial of this matrix, then:
  \begin{equation*}
    x^{\ell+1}\pi(-x-x^{-1})=(x+1)(x^{n-1}+\dots+x+1).
  \end{equation*}
\end{theorem}

\subsection{Sketch of proof of Theorem \ref{theorem:only-applies}}
\label{sec:sketch-proof-theorem-only-applies}

To handle the general case of $\mathrm{Gr}(k,n)$ for $k> 2$, we will describe the quantum cohomology in a more abstract fashion. It is known that $x_{1},\dots,x_{k}$ generate the ordinary cohomology (this is a consequence of the \emph{Giambelli formula}). Because the Giambelli formula still holds with the quantum product, see \cite{bertram-aim-1997,buch-compositio-2003}, it follows that $x_{1},\dots,x_{k},q$ still generate the algebra:
\begin{equation*}
  H^{*}(\mathrm{Gr}(k,n),\Z)\otimes \Z[q]
\end{equation*}
with respect to the quantum product. Thus one can consider the quantum cohomology as a quotient:
\[
  H^{*}(\mathrm{Gr}(k,n),\Z)\otimes \Z[q]=\Z[x_{1},\dots,x_{k},q]/\mathscr{I},
\]
where $\mathscr{I}$ is some ideal of the polynomial ring; \cite{siebert-tian-AJM-1997} establishes:
\begin{lemma}\label{lemma:ideal-description}
The ideal $\mathscr{I}$ is generated by:
  \begin{equation*}
    Y_{n-k+1},\dots,Y_{n-1},Y_{n}+(-1)^{k}q
  \end{equation*}
  where $Y_r\in \Z[x_{1},\dots,x_{k}]$ are the degree $r$ polynomials defined by:
  \begin{equation}
    \label{eq:recursion}
    Y_{r}=\det(x_{1+j-i})_{1\le i,j\le r}.
  \end{equation}
\end{lemma}
\begin{proof}
  This is proved in \cite[\S8]{buch-compositio-2003} using the Pieri formula.
\end{proof}

Morally, $Y_r$ represents the Schubert class $\sigma_D$ where $D$ has $r$ boxes in the top row (of course, this description does not make sense for $r > n-k$). Following \cite{galkin-golyshev-RMS-2006}, we use this presentation to establish the following theorem.

Before we state the theorem, let us introduce the notion of an \emph{admissible collection of $n$th roots of unity} in nonzero characteristic. Let $\mathrm{char}(\mathbf{F})=p$ and suppose $n=p^{d}m$. Let us say that a multiset $J=\set{\zeta_{1},\dots,\zeta_{k}}$ of $n$th roots of unity is \emph{admissible} provided each root appears with at most multiplicity $p^{d}$. In this case, the complementary multiset $J^{c}$ is well-defined, so that $J\cup J^{c}$ contains all the $n$th roots of unity, each with their multiplicity $p^{d}$. If $\mathrm{char}(\mathbf{F})=0$, an admissible multiset is just a set (each element has multiplicity at most $1$). In the following, we fix $\xi$ solving $\xi^{n}+(-1)^{k}=0$, and let $\mathbf{K}$ denote the splitting field over $\mathbf{F}$ for the polynomial $x^{n}+(-1)^{k}$ (note that this field also contains the $n$th roots of unity).

\begin{theorem}\label{theorem:ring-homomorphism}
  For any admissible multiset $J=\set{\zeta_{1},\dots,\zeta_{k}}$ of $n$th roots of unity, there is a ring homomorphism:
  \begin{equation*}
    \mathrm{ev}_{J}:\mathrm{QH}^{*}(\mathrm{Gr}(k,n);\mathbf{F})\to \mathbf{K}
  \end{equation*}
  sending $q$ to $1$ and $x_{i}$ to the expression:
  \begin{equation*}
    \mathrm{ev}_{J}(x_{i})=\xi^{i}\sum \zeta_{1}^{i_{1}}\dots\zeta_{k}^{i_{k}},
  \end{equation*}
  where the sum is over all $i=i_{1}+\dots+i_{k}$ such that each $i_{j}$ is $0$ or $1$ (i.e., is given by evaluating the elementary symmetric polynomials).
\end{theorem}
We recall the proof in \S\ref{sec:proof-theorem-ring-homomorphism} (it follows the same argument as \cite{galkin-golyshev-RMS-2006}). This result is used to prove Theorem \ref{theorem:only-applies} in \S\ref{sec:proof-theorem-only-applies}.

\subsection{Lagrangian intersections}
\label{sec:lagr-inters}

Suppose that $L\subset M$ is a compact connected monotone Lagrangian submanifold with minimal Maslov number\footnote{The Maslov number is an integer associated to a smooth disk in $M$ with boundary on $L$, and it governs dimensions of moduli spaces of holomorphic disks $u:(D,\bd D)\to (M,L)$.} at least $2$ (all Lagrangians will be assumed to satisfy these assumptions, in this paper). The \emph{Lagrangian quantum cohomology} $\mathrm{QH}^{*}(L,\mathbf{F})$, as defined in \cite{biran-cornea-rigidity-uniruling-2009} using the pearl (aka cluster) approach of \cite{cornea-lalonde-arXiv-2005}, gives a tool to probe the quantum cohomology of the ambient space $M$. The approach of \cite{biran-cornea-rigidity-uniruling-2009}, its sequel \cite{biran-cornea-GT-2012}, and the later work \cite{leclercq-zapolsky-jta-2018,zapolsky-arXiv-2015} produces a graded module over $\mathbf{F}[\lambda^{-1},\lambda]$ where $\lambda$ has grading equal to the minimal Maslov number of $L$. In fact, $\mathrm{QH}^{*}(L,\mathbf{F})$ is the homology of the Morse complex $\mathrm{CM}^{*}(L,\mathbf{F})\otimes\mathbf{F}[\lambda^{-1},\lambda]$ with a deformed differential: $$d=d_{0}+\lambda d_{1}+\lambda^{2}d_{2}+\dots$$ defined by counting ``quantum trajectories'' involving chains of Morse flow lines and holomorphic disks; see Figure \ref{fig:quantum-trajectory}.

In the case $\mathrm{char}(\mathbf{F})\ne 2$, we assume that our Lagrangians are \emph{spin}.\footnote{Here ``spin'' means the structure group of $TL$ can be ``reduced'' to the spin group covering $SO(n)\subset O(n)$ --- this is used for defining signs in various identities.} This is the assumption that \cite{biran-cornea-GT-2012} use (we will appeal to some of their results below); however, we refer the reader to \cite[Definition 18]{leclercq-zapolsky-jta-2018} and \cite{zapolsky-arXiv-2015} for more general hypotheses.\footnote{For our purposes, we will only require knowing that tori are spin, and that the quaternionic Grassmannians $\mathrm{Gr}_{\mathbf{H}}(k,n)$ are spin (which is obvious after appealing to the fact that manifolds are spin if $w_{1}(TL)$ and $w_{2}(TL)$ both vanish; see \cite{lawson-michelsohn-PUP-1989}).} Their framework gives $\mathrm{QH}^{*}(L,\mathbf{F})$ the structure of an
algebra\footnote{To be precise, in general $\mathrm{QH}^{*}(M,\mathbf{F})$ is a ``supercommutative'' algebra, and $\mathrm{QH}^{*}(L,\mathbf{F})$ has the structure of a ``superalgebra'' over $\mathrm{QH}^{*}(M,\mathbf{F})$; see \cite{zapolsky-arXiv-2015}. This distinction is irrelevant when $M=\mathrm{Gr}(k,n)$ since its quantum cohomology vanishes in odd degrees.} over the ambient quantum cohomology $\mathrm{QH}^{*}(M,\mathbf{F})$:
\begin{equation}\label{eq:module-action}
  \mathrm{QH}^{*}(M,\mathbf{F})\otimes \mathrm{QH}^{*}(L,\mathbf{F})\to \mathrm{QH}^{*}(L,\mathbf{F}),
\end{equation}
with a unit element $1_{L}\in \mathrm{QH}^{0}(L,\mathbf{F})$; see \cite[Lemma 5.2.4 and \S5.3]{biran-cornea-arXiv-2007}.\footnote{The module action is such that $q$ acts by $\lambda^{k}$ where $k\omega(u)=1$ where $u$ is some generator of $\pi_{2}(M,L)$; here $\omega(\pi_{2}(M))=1$, so $\omega(\pi_{2}(M,L))=1/k$.} The following combines results from \cite{albers-extrinsic-topology-2005,biran-cornea-rigidity-uniruling-2009}, \cite[\S8]{entov-polterovich-compositio-2009}, and \cite{leclercq-zapolsky-jta-2018}, and concerns the spectral invariants recalled in \S\ref{sec:floer-cohomology-review}.
\begin{theorem}\label{theorem:lagrangian}
  Fix a field $\mathbf{F}$, let $L,1_{L}$ be as above, and assume $e\ast 1_{L}\ne 0$. Then the following ``Lagrangian control property'' holds:
  \begin{equation}\label{eq:lagrangian-control}
    c(e,H_{t})\le c(e,0)+\int_{0}^{1}\max_{L}H_{t}dt.
  \end{equation}
  In particular given two such data $L,1_{L}$ and $K,1_{K}$, if:
  \begin{equation}\label{eq:lagrangian-criterion}
    \text{$1_{L}\ne 0$ and $1_{K}\ne 0$ for disjoint Lagrangians $L,K$}
  \end{equation}
  then the diameter of $\gamma$ is infinite. The same coefficient field must be used for both Lagrangians.
\end{theorem}
The combination of Theorem \ref{theorem:main} and Theorem \ref{theorem:lagrangian} can sometimes be used to force intersections between Lagrangians with non-vanishing Lagrangian quantum cohomology.\footnote{An alternative mechanism (explained to the authors by Egor Shelukhin) for forcing Lagrangian intersections, which bypasses spectral invariants, utilizes the structural properties of the open-closed maps and the Cardy relation (which analyzes moduli spaces of holomorphic annuli with boundaries on two Lagrangians). See, e.g., the generation criterion of \cite{abouzaid-IHES-2010}, its adaptation to the monotone setting in \cite{sheridan-pubihes-2016}, and the work of \cite{seidel-invent-2014} on disjoinable Lagrangian spheres.}

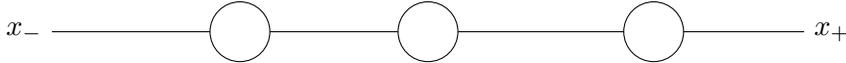
\begin{figure}[h]
  \centering
  \begin{tikzpicture}
    \draw[every node/.style={fill=black,circle,inner sep=1pt}] (0,0)--node[pos=0.25](A){}node[pos=0.5](B){}node[pos=0.8](C){}(10,0);
    \draw[every node/.style={fill=white,draw,circle,inner sep=8pt}] (A)node{} (B)node{} (C)node{};
    \node at (0,0) [left]{$x_{-}$};
    \node at (10,0) [right]{$x_{+}$};
  \end{tikzpicture}
  \caption{A quantum trajectory contributing to the deformation of the Morse differential of $L$.}
  \label{fig:quantum-trajectory}
\end{figure}

\subsubsection{Real and quaternionic Grassmannians as Lagrangians}
\label{sec:real-and-quaternionic}

Two classes of Lagrangians in complex Grassmannians are the real and quaternionic Grassmannians (see, e.g., \cite{zapolsky-dedicata-2020,kislev-shelukhin-GT-2021}). We have:
\begin{prop}\label{proposition:real-grassmannian}
  Consider the real Lagrangian submanifold $L=\mathrm{Gr}_{\mathbf{R}}(k,n)$, and suppose $\mathrm{char}(\mathbf{F})=2$. Then $\mathrm{QH}^{*}(\mathrm{Gr}_{\mathbf{\R}}(k,n),\mathbf{F})\ne 0$.
\end{prop}
\begin{prop}\label{proposition:quaternionic-grassmannian}
  The quaternionic Lagrangian submanifold $L=\mathrm{Gr}_{\mathbf{H}}(k,n)$ in $M=\mathrm{Gr}(2k,2n)$ satisfies $\mathrm{QH}^{*}(L,\mathbf{F})\ne 0$ for every field $\mathbf{F}$.
\end{prop}
To show that these are indeed Lagrangians, we comment that both are \emph{fixed point sets of antisymplectic involutions}.\footnote{An involution is \emph{antisymplectic} provided it pulls back $\omega$ to $-\omega$. The submanifold of fixed point of such an involution is necessarily a Lagrangian submanifold.} For $L=\mathrm{Gr}_{\R}(k,n)$, one simply conjugates all entries of a full-rank matrix in $\C^{n\times k}$, and this antisymplectic involution descends to an involution of $\mathrm{Gr}(k,n)$ fixing $L$. On the other hand, for the quaternionic Lagrangian, one considers the involution $\mathfrak{J}$ obtained by complex conjugating and then multiplying by:
\begin{equation}\label{eq:quaternion-half}
  \mathrm{diag}([
  \begin{smallmatrix}
    0&-1\\
    1&0
  \end{smallmatrix}
  ],\dots,[
  \begin{smallmatrix}
    0&-1\\
    1&0
  \end{smallmatrix}
  ]),
\end{equation}
this is also an antisymplectic map descending to an involution of $\mathrm{Gr}(2k,2n)$ fixing $\mathrm{Gr}_{\mathbf{H}}(k,n)$.

To see why the fixed points of the latter involution forms $\mathrm{Gr}_{\mathbf{H}}(k,n)$, we digress for a bit and explain the set-up. Let $\mathbf{H}=\set{a+bi+c\mathfrak{J}+di\mathfrak{J}}$ act on the complex vector space $\mathbf{C}^{2n}$ in such a way that:
\begin{itemize}
\item $\mathfrak{J}$ acts by complex conjugation followed by \eqref{eq:quaternion-half}.
\end{itemize}
Then any plane $P\in \mathrm{Gr}_{\mathbf{H}}(k,n)$, considered as a subset of $\C^{2n}=\mathbf{H}^{n}$, can be given a basis as a $2k$ dimensional complex subspace, of the form $X_{1},\mathfrak{J}X_{1},\dots,X_{k},\mathfrak{J}X_{k}$. Thus, if we complex conjugate and then apply \eqref{eq:quaternion-half}, we simply send this basis to $\mathfrak{J}X_{1},-X_{1},\dots,\mathfrak{J}X_{n},-X_{n}$, which spans the same plane. Thus $\mathrm{Gr}_{\mathbf{H}}(k,n)$ lies in the fixed point set of the involution. Conversally, if a plane $P$ lies in the fixed point set, then it is not hard to show it admits a basis of the form $X_{1},\mathfrak{J}X_{1},\dots,X_{n},\mathfrak{J}X_{n}$; the desired result follows.

Except for one edge case, the propositions on the non-vanishing of the Lagrangian quantum cohomologies follow from:
\begin{lemma}[{\cite[Theorem 1.2.2]{biran-cornea-rigidity-uniruling-2009}}]\label{lemma:BC-lemma}
  If the minimal Maslov number $N_{L}$ of a monotone Lagrangian is large enough that $[\mathrm{pt}]$ can be expressed as a cup product of classes in $H^{d_{1}}(L,\mathbf{F}),H^{d_{2}}(L,\mathbf{F}),\dots$ where $d_{i}\le N_{L}-2$, then $\mathrm{QH}^{*}(L,\mathbf{F})\ne 0$. Here we need $\mathrm{char}(\mathbf{F})=2$ if $L$ is not spin to ensure $[\mathrm{pt}]\ne 0$ (as well as ensure the Lagrangian quantum cohomology is well-defined).\hfill$\square$
\end{lemma}
Using this Lemma we can prove the two propositions.
\begin{proof}[Proof of Proposition \ref{proposition:real-grassmannian}]
  The minimal Maslov number $N_{L}$ of $L=\mathrm{Gr}_{\mathbf{R}}(k,n)$ is known to be $n$ (because $H_{1}(L)=\Z/2\Z$ is $2$-torsion, the minimal Maslov number is either $n$ or $2n$, since $n$ is the minimal Chern number of the ambient space; with a bit of work one proves it must be $n$). However, $[\mathrm{pt}]$ can be expressed as the cup-product of classes in degree $\max\set{k,n-k}$; see Remark \ref{remark:cup-product-projective-space}. So, except for the edge case $\mathrm{Gr}(1,2)$, we can apply Lemma \ref{lemma:BC-lemma}. The result is known to hold in this edge case.
\end{proof}
\begin{proof}[Proof of Proposition \ref{proposition:quaternionic-grassmannian}]
  The minimal Maslov number $N_{L}$ of $L=\mathrm{Gr}_{\mathbf{H}}(k,n)$ is $4n$ because it is simply connected (the Maslov number is twice the Chern number of the ambient space). On the other hand, $[\mathrm{pt}]$ can be expressed as the cup-product of classes in degree $4\max\set{k,n-k}$; see Remark \ref{remark:cup-product-projective-space}.
\end{proof}
\begin{remark}\label{remark:cup-product-projective-space}
  In both proofs, we appeal to some result about $[\mathrm{pt}]$ being generated by classes of a certain degree. In this remark, we provide an explicit argument. Without loss of generality, suppose $k\le 2n$. Then this follows from the relation $\sigma^{n-k}=\mathrm{PD}(\mathrm{pt})$ where $\sigma$ is the class Poincaré dual to the cycle of $k$-planes lying in a subspace of codimension $1$. Indeed, the cup product $\sigma^{n-k}$ is Poincaré dual to the cycle of $k$-planes contained in a subspace of codimension $n-k$; but there is only one such $k$-plane, since a subspace of codimension $n-k$ has dimension $k$. Thus $\sigma^{n-k}$ is Poincaré dual to a point.
\end{remark}

See \cite{oh-CPAM-1993,iriyeh-sakai-tasaki-JMSJ-2013} for general non-vanishing results for the Floer cohomologies of fixed point sets of antisymplectic involutions.

\subsubsection{B-fields, local systems, and monotone tori}
\label{sec:b-fields-monotone}

There is a large and well-studied class of Lagrangians which does not fall under the umbrella of Lemma \ref{lemma:BC-lemma}, namely, the \emph{monotone torus fibres} of involutive maps:\footnote{Recall that a map $\mathfrak{Z}$ is involutive if all Hamiltonians of the form $h\circ \mathfrak{Z},k\circ \mathfrak{Z}$ generate commuting vector fields.} $$\mathfrak{Z}:M^{2n}\to \R^{n}.$$ We require that $\mathfrak{Z}$ is a submersion in a neighborhood of some value $u\in \R^{n}$, so that the inverse image $T=\mathfrak{Z}^{-1}(u)$ is a Lagrangian torus. Such a torus does not satisfy the hypotheses of Lemma \ref{lemma:BC-lemma} if $N_{T}=2$. If one requires assumptions on the singularities $\mathfrak{Z}$, e.g., if $\mathfrak{Z}$ generates a Hamiltonian $T^{n}$-action, then one automatically has $N_{T}=2$; such tori are quite important in mirror symmetry \cite{fukaya-mirror-2001,auroux-JGGT-2007,sheridan-pubihes-2016} and have been extensively studied, see, e.g., \cite{cho-IMRN-2004,cho-oh-AJM-2006,FOOO-book-2009,FOOO-Duke-2010,FOOO-MEMO-2019}. See also \cite{nishinou-nohara-ueda-adv-math-2010,cho-kim-imrn-2021,castronovo-thesis-2021,castronovo-QT-2023} for results specific to Grassmannians. These references all explore deformed versions of the Floer cohomology.

The most general set-up we would like to consider is the \emph{Lagrangian quantum cohomology with a $B$-field}, related to the work of \cite{cho-JGP-2008}.\footnote{See also \cite{fukaya-CM-2002}, and also \cite{usher-GT-2011,ritter-GT-2009} for related work in the closed string case.} Here we fix a cohomology class $B\in H^{2}(M,L,\mathbf{F}^{\times})$ and use it to deform the counts appearing in the definition of:
\begin{itemize}
\item $\mathrm{QH}^{*}(M,\mathbf{F})$ (the quantum cup product),
\item $\mathrm{QH}^{*}(L,\mathbf{F})$ (the quantum differential and module action),
\item Hamiltonian Floer cohomology $\mathrm{HF}^{*}(H_{t},\mathbf{F})$; see \S\ref{sec:floer-cohomology-review}.
\end{itemize}
Denote the deformed objects by $\mathrm{QH}^{*}_{B}(L,\mathbf{F})$, $\mathrm{QH}^{*}_{B}(M,\mathbf{F})$, and $\mathrm{HF}^{*}_{B}(H_{t},\mathbf{F})$. One concludes the existence of $B$-deformed spectral invariants $c_{B,\mathbf{F}}(e,H_{t})$ for classes $e\in \mathrm{QH}^{*}_{B}(M,\mathbf{F})$; we refer the reader to \cite{usher-GT-2011} for an overview and for further motivation.

Let us explain the ideas a bit more precisely. To keep things simple, let us suppose that\footnote{In this case, we should work with a field extension of $\mathbf{F}=\Z/2\Z$, otherwise $\mathbf{F}^{\times}$ is trivial. One can take, e.g., the splitting field of $x^{3}+1$ which has $\mathbf{F}^{\times}=\set{1,\zeta,\zeta^{2}}$.} $\mathrm{char}(\mathbf{F})=2$. We define the $B$-deformed differential by:
\begin{equation*}
  d(x_{-})=\sum_{u} B(u)q^{\omega(u)}x_{+}
\end{equation*}
where the sum is over quantum trajectories starting at $x_{-}$ weighted only by the total symplectic area $\omega(u)$ and $B(u)\in \mathbf{F}^{\times}$; here $B(u)=B(u_{1})\dots B(u_{\ell})$, if $u_{1},\dots,u_{\ell}$ are the pearls appearing in a given quantum trajectory. Then:
\begin{prop}\label{proposition:B-deformed}
  The $B$-deformed differential still squares to zero. The module action $\ast$ of $B$-deformed quantum cohomology can be defined as in the undeformed case. Moreover: $$c_{B}(e,H_{t})\le c_{B}(e,0)+\int_{0}^{1} \max H_{t}|_{L}dt$$ holds provided that $1_{L,B}\ast e$ is non-zero.
\end{prop}
\begin{proof}[Proof]
  The argument is actually just a special case of the general results in \cite{leclercq-zapolsky-jta-2018}, especially those concerning twisted coefficients: \cite[\S2.5.2]{leclercq-zapolsky-jta-2018}. For related results we refer the reader to \cite{FOOO-MEMO-2019}.
\end{proof}

This theory is well developed in characteristic zero, in particular over the field $\mathbf{F}=\C$, and fits under the rubric of \emph{bulk-deformed Floer cohomology}; see \cite{FOOO-MEMO-2019} for an advanced treatment. We also refer the reader to the recently posted \cite{sheridan2025constructingbigrelativefukaya}.\footnote{Expanding the full theory of ``bulk deformations'' of \cite{FOOO-MEMO-2019} to non-zero characteristics appears to be quite complicated story due to factors like $1/k!$ which appear in many definitions; see \cite{sheridan2025constructingbigrelativefukaya} for some discussion of this.} Here, one sometimes restricts to those $B$-fields which admit a logarithm; see, e.g., \cite[Remark 3.5]{auroux-JGGT-2007}.

One can consider $B$ in the image of the map $H^{1}(L,\mathbf{F}^{\times})\to H^{2}(M,L,\mathbf{F}^{\times})$. These special $B$-fields satisfy $c_{B}(e,H_{t})=c(e,H_{t})$, i.e., \emph{they do not change the closed string spectral invariants} --- in the long exact sequence in cohomology, such $B$ map to the trivial class in $H^{2}(M,\mathbf{F}^{\times})$. These are often called \emph{local systems} in the literature; see, e.g., \cite[\S 2.4]{biran-cornea-GT-2012}. One has:
\begin{corollary}\label{corollary:infinite-spectral-diameter-local-system}
  Fix a field $\mathbf{F}$. If $1_{L,B}\ne 0$ and $1_{K,B'}\ne 0$ for local systems $B,B'$, where $L,K$ are disjoint Lagrangians satisfying the monotonicity assumptions from the start of \S\ref{sec:lagr-inters}, then the diameter of $\gamma$ is infinite.
\end{corollary}
\begin{proof}
  The argument is the same as the one used for Theorem \ref{theorem:lagrangian}, and uses the Lagrangian control property for both Lagrangians to achieve a large spectral norm.
\end{proof}

\subsubsection{Disk potentials}
\label{sec:disk-potentials}

For monotone tori $T^{n}\subset M^{2n}$ with minimal Maslov number $2$, there is an established strategy for finding $B\in H^{1}(T,\mathbf{F}^{\times})$ that yield non-vanishing quantum cohomology $\mathrm{QH}^{*}_{B}(T,\mathbf{F})$: {\itshape critical points $B$ of a disk counting function yield $\mathrm{QH}^{*}_{B}(T,\mathbf{F})\ne 0$}; see \cite[Proposition 3.3.1]{biran-cornea-GT-2012} and \cite{cho-oh-AJM-2006}. The idea of disk counting functions (or \emph{potentials}) originates from physics \cite{hori2000mirrorsymmetry}. Briefly, the disk potential $W(z)$ counts $J$-holomorphic disks with Maslov number 2 passing through a point, and records the homology classes traced out by the boundary of the holomorphic disks in a Laurent polynomial:
\begin{equation*}
  W_{T,J,\mathrm{pt}}(z_{1},\dots,z_{n})=\sum_{u} z_{1}^{k_{1}(\bd u)}\dots z_{n}^{k_{n}(\bd u)}.
\end{equation*}
To be more precise, the sum is over $J$-holomorphic maps $u:D\to (M,T)$ with $u(\infty)=\mathrm{pt}$ and $\mu(u)=2$, modulo reparametrization by the two-dimensional group of biholomorphisms preserving $\infty$ (here we think of $D$ as compactifying the upper half-plane). The integer weights $k_{i}(\bd u)$ is the degrees of the map $q_{i}\circ u|_{\bd D}:\bd D\to \R/\Z$ (here $q_{1},\dots,q_{n}$ are coordinates on the torus). These counts are independent of $(J,\mathrm{pt})$, provided they are chosen generically, and we write $W_{T}(z_{1},\dots,z_{n})$. This invariant detects which local systems give a non-vanishing Lagrangian quantum cohomology:
\begin{theorem}
  A local system $B\in H^{1}(T,\mathbf{F}^{\times})$ satisfies:
  \begin{equation*}
    \d W_{T}(z_{1},\dots,z_{k})=0
  \end{equation*}
  where $z_{i}=B(e_{i})$, if and only if $\mathrm{QH}^{*}_{B}(T,\mathbf{F})\ne 0$; here $e_{i}$ is the loop parametrizing the $i$ direction.\footnote{Let us note that, while $W_{T}$ and $z_{i}=B(e_{i})$ both depend on the choice of identification $(\R/\Z)^{n}\simeq T$, the vanishing of $\d W_{T}(z_{1},\dots,z_{k})$ is independent of this choice.}
\end{theorem}
\begin{proof}
  The result is well-known, see \cite{cho-oh-AJM-2006,auroux-JGGT-2007,FOOO-Duke-2010}, \cite[\S3.3]{biran-cornea-GT-2012}, \cite[\S7.2]{evans-lekili-JAMS-2019}, and \cite{castronovo-adv-math-2020}. The argument we have in mind is more or less exactly what is written in the proof of \cite[Proposition 3.3.1]{biran-cornea-GT-2012} (in the case of tori). We summarize the argument with our conventions: one shows that the Laurent polynomial:
  \begin{equation*}
    z_{j}\pd{W_{T}}{z_{j}},
  \end{equation*}
  evaluated at the special points $z_{1}=B(e_{1}),z_{2}=B(e_{2}),\dots$, gives exactly the coefficient in the quantum differential joining the class $e_{j}^{\vee}$ dual to $e_{j}$ in $H^{1}(T,\mathbf{F})$ to the class of the unit $1\in H^{0}(T,\mathbf{F})$. Here $e_{j}^{\vee}$ and $1$ are represented by critical points of a perfect Morse function. If these coefficients are all zero, it implies that $e_{j}^{\vee}$ are quantum cycles, and hence:\footnote{Here $\ast$ is the quantum product, a deformation of the cup product; see \cite[\S5.1]{biran-cornea-rigidity-uniruling-2009}.}
  \begin{equation*}
    e_{1}^{\vee}\ast \dots\ast e_{n}^{\vee}=\mathrm{PD}(\mathrm{pt})+qA_{1}+\dots
  \end{equation*}
  is a quantum Morse cycle. However, such a quantum Morse cycle cannot be exact, by degree reasons (and the fact $\mathrm{PD}(\mathrm{pt})$ is not exact in the classical Morse complex). This completes the proof.
\end{proof}

\subsubsection{The Gelfand-Cetlin torus}
\label{sec:gelf-cetl-syst}

There are some monotone Lagrangian tori in $\mathrm{Gr}(k,n)$ whose disk potentials are computed. The original argument is due to \cite{nishinou-nohara-ueda-adv-math-2010}, which computes the potential of the monotone fibre $T=\mathfrak{Z}^{-1}(u)$ of the \emph{Gelfand-Cetlin system} (an involutive map): $$\mathfrak{Z}:\mathrm{Gr}(k,n)\to \R^{k(n-k)};$$ see \cite{gelfand-cetlin-dokl-1950,guillemin-sternberg-jfa-1983} and \S\ref{sec:disjoint-lagrangians}. Each component function $\mathfrak{Z}_{i,j}$ is smooth and generates a Hamiltonian circle action on an open and dense subset of $\mathrm{Gr}(k,n)$ containing the torus $T$; see \cite[Theorem 4.3]{cho-kim-imrn-2021} and \cite[\S 3]{castronovo-adv-math-2020}. In \cite[Proposition 3.5]{castronovo-adv-math-2020}, a formula is given for the potential:\footnote{For the $u$ such that $T=\mathfrak{Z}^{-1}(u)$, we refer the reader to \cite{cho-kim-imrn-2021} and \cite{castronovo-adv-math-2020}.}
\begin{equation}\label{eq:castronovo-potential} W_{T}(z)=\sum_{i=1}^{k-1}\sum_{j=1}^{n-k}\frac{z_{i,j}}{z_{i+1,j}}+\sum_{i=1}^{k}\sum_{j=1}^{n-k-1}\frac{z_{i,j+1}}{z_{i,j}}+\frac{1}{z_{1,n-k}}+z_{k,1},
\end{equation}
where $B(e_{i,j})=z_{i,j}$ and where $e_{i,j}$ is the orbit class of the circle action generated by $\mathfrak{Z}_{i,j}$. His deduction depends on work of \cite{nishinou-nohara-ueda-adv-math-2010,an-cho-kim-ejc-2018}.

\begin{lemma}\label{lemma:real-valued-critical-point}
  There is a point $z\in (0,\infty)^{k(n-k)}$ such that $\d W_{T}(z)=0$. In particular, $\mathrm{QH}_{B}^{*}(T,\mathbf{R})\ne 0$ for some real-valued local system $B\in H^{1}(T,\mathbf{R}^{\times})$.
\end{lemma}
\begin{proof}
  This is due to \cite{castronovo-adv-math-2020}, where explicit formulas are found for the coordinates $z_{i,j}$ in terms of the evaluations of certain symmetric polynomials at collections of $n$th roots of unity. The case of positive real coefficients is rather straightforward (once one has the formula for potential, of course). One minimizes $W_{T}(z)\in (0,\infty)$ on the open locus where $z_{i,j}\in (0,\infty)$. The minimum is attained: indeed, if along some minimizing sequence one has $z_{i,j}\to 0$, then $z_{i-1,j}$ or $z_{i,j+1}$ must also tend to zero (otherwise one of the terms would explode). In this manner, one concludes that $z_{1,n-k}$ must tend to zero, which shows $W_{T}(z)\to +\infty$, contradicting the fact the sequence was minimizing. Similarly, if $z_{i,j}$ tends to $+\infty$ along a minimizing sequence, one concludes $z_{k,1}\to +\infty$ as well, again yielding a contradiction that the sequence was minimizing for $W_{T}(z)$.
\end{proof}

\subsubsection{Disjoint Lagrangians and infinite spectral diameter}
\label{sec:disjoint-lagrangians}

We prove:
\begin{lemma}\label{lemma:disjoint-lagrangians}
The quaternionic Lagrangian $L=\mathrm{Gr}_{\mathbf{H}}(k,n)$ is disjoint from the Gelfand-Cetlin torus $T$ inside $\mathrm{Gr}(2k,2n)$, provided that $k<n$.
\end{lemma}
\begin{proof}
  The argument\footnote{In the first draft of this paper, we argued based on the \emph{toric degeneration} described in \cite{nishinou-nohara-ueda-adv-math-2010}, in a manner inspired by \cite[\S3]{kawamoto-math-ann-2024}. The current proof is more elementary.} we will present relies on the formula for the Gelfand--Cetlin system $\mathfrak{Z}$, which we recall now. Let us describe the set up for $\mathrm{Gr}(k,n)$; we will specialize to the case $\mathrm{Gr}(2k,2n)$ later in the proof.

  Each subspace $P\in \mathrm{Gr}(k,n)$ can be represented as the image of an unitary matrix $U$ of size $n\times k$ so that the columns of $U$ form a unitary basis for $P$. The endomorphism:
  \begin{equation*}
    A=UU^{\dagger}:\C^{n}\to \C^{n}
  \end{equation*}
  has spectrum $\set{1,0}$, with the eigenvalue $1$ having multiplicity $k$ and the eigenvalue $0$ having multiplicity $n-k$. This endomorphism depends only on the subspace $P$. Let $A^{(r)}$ be the upper left $r\times r$ submatrix of $A$, and let:
  \begin{equation*}
    \lambda^{(r)}_{1}\ge \dots \ge \lambda^{(r)}_{r}
  \end{equation*}
  be the eigenvalues of $A^{(r)}$, considered as real-valued functions on $\mathrm{Gr}(k,n)$. Many of these eigenvalues are constant on $\mathrm{Gr}(k,n)$. For instance, it can be shown using minimax theory for eigenvalues that:
  \begin{equation}\label{eq:eigenvalue-minmax}
    \lambda^{(r+1)}_{i}\ge \lambda^{(r)}_{i}\ge \lambda^{(r+1)}_{i+1}\text{ for }1\le r \le n-1\text{ and }i\le r.
  \end{equation}
  In particular, since:
  \begin{equation*}
    \text{$\lambda_{1}^{(n)}=\dots=\lambda_{k}^{(n)}=1$ and $\lambda_{k+1}^{(n)}=\dots=\lambda_{n}^{(n)}=0$,}
  \end{equation*}
  it holds automatically that:
  \begin{itemize}
  \item $\lambda_{i}^{(r)}=1$ if $i+n-r\le k$,
  \item $\lambda_{i}^{(r)}=0$ if $i\ge k+1$.
  \end{itemize}
  This leaves $k(n-k)$ many choices for $i,r$ which are not constant functions $0$ or $1$. Indeed, these $k(n-k)$ many remaining eigenvalues form the component functions of the Gelfand-Cetlin map $\mathfrak{Z}:\mathrm{Gr}(k,n)\to \R^{k(n-k)}$; see \cite{nishinou-nohara-ueda-adv-math-2010}.

  Following \cite{castronovo-adv-math-2020}, let us relabel:
  \begin{equation*}
    \mathfrak{Z}_{i,j}=\lambda^{(i+j-1)}_{i},
  \end{equation*}
  so that $\mathfrak{Z}_{i,j}$, $i=1,\dots,k$ and $j=1,\dots,n-k$, are the components of $\mathfrak{Z}$. Clearly, the image of $\mathfrak{Z}$ is contained in the polytope determined by the inequalities \eqref{eq:eigenvalue-minmax}, which amounts to:
  \begin{equation}\label{eq:GC-inequalities}
    \mathfrak{Z}_{i,j+1}\ge \mathfrak{Z}_{i,j}\ge \mathfrak{Z}_{i+1,j}.
  \end{equation}
  In fact, the image of $\mathfrak{Z}$ is precisely the set of tuples $\mathfrak{Z}_{i,j}$ satisfying the inequalities \eqref{eq:GC-inequalities}; see \cite[Lemma 3.5]{nishinou-nohara-ueda-adv-math-2010}, \cite{castronovo-adv-math-2020}, and the references therein. It will be important for us that:
  \begin{enumerate}[label=(\alph*)]
  \item\label{GC-polytope-1} the locus where $\mathfrak{Z}_{1,2}=\mathfrak{Z}_{2,1}$ lies in the boundary of the polytope,
  \end{enumerate}
  provided that $n,k\ge 2$. It will also be important that the Gelfand-Cetlin torus $T=\mathfrak{Z}^{-1}(u)$ of \S\ref{sec:gelf-cetl-syst} is such that $u$ lies in the interior of the polytope (see, e.g., \cite[Theorem 3.2]{castronovo-adv-math-2020}), that is to say, using \ref{GC-polytope-1}:
  \begin{enumerate}[label=(\alph*),resume]
  \item\label{GC-polytope-2} $\mathfrak{Z}_{1,2}>\mathfrak{Z}_{2,1}$ holds on $T$.
  \end{enumerate}
  To complete the proof, we specialize to the case of $\mathrm{Gr}(2k,2n)$, and claim:
  \begin{enumerate}[label=(\alph*),resume]
  \item\label{GC-polytope-3} $\mathfrak{Z}_{1,2}=\mathfrak{Z}_{2,1}$ holds identically on $\mathrm{Gr}_{\mathbf{H}}(k,n)$.
  \end{enumerate}
  The combination of \ref{GC-polytope-2} and \ref{GC-polytope-3} yields the desired result. The remainder of the proof is dedicated to establishing \ref{GC-polytope-3}. The idea is that any subspace in $\mathrm{Gr}_{\mathbf{H}}(k,n)$ is the image of $U$ of the form:
  \begin{equation*}
    U=
    \left[
      \begin{matrix}
        z_{1}&-\bar{z}_{2}&z_{3}&-\bar{z}_{4}&\cdots&z_{2k-1}&-\bar{z}_{2k}\\
        z_{2}&\bar{z}_{1}&z_{4}&\bar{z}_{3}&\cdots&z_{2k}&\bar{z}_{2k-1}\\
        \vdots&\vdots&\vdots&\vdots&\ddots&\vdots&\vdots\\
      \end{matrix}\right],
  \end{equation*}
  where $U=[X_{1},\mathfrak{J}X_{1},X_{2},\mathfrak{J}X_{2},\dots]$ as described in \S\ref{sec:real-and-quaternionic}, and we suppose these vectors form an orthonormal basis (note that we only require the first two rows of $U$ to describe $\mathfrak{Z}_{2,1}$ and $\mathfrak{Z}_{1,2}$). With $A=UU^{\dagger}$, we have:
  \begin{equation*}
    A^{(2)}=\left[
      \begin{matrix}
        \sum \abs{z_{i}}^{2}&0\\
        0&\sum \abs{z_{i}}^{2}
      \end{matrix}
    \right];
  \end{equation*}
  in particular, the two eigenvalues of $A^{(2)}$ are equal, that is to say $\mathfrak{Z}_{2,1}=\mathfrak{Z}_{1,2}$. This establishes \ref{GC-polytope-3}, and finishes the proof of the lemma.
\end{proof}

With this result established, as well as the earlier results Proposition \ref{proposition:quaternionic-grassmannian} and Lemma \ref{lemma:real-valued-critical-point} on the non-vanishing of the Lagrangian quantum cohomology for $L$ and $T$, we can appeal to Corollary \ref{corollary:infinite-spectral-diameter-local-system} to conclude that the spectral diameter of $\mathrm{Gr}(2k,2n)$ is infinite; this yields Theorem \ref{theorem:2k-2n}.

\subsection{Further questions}
\label{sec:further-questions}

There are two competing principles at play:
\begin{enumerate}[label=(\alph*)]
\item\label{principle-a} If the quantum cohomology over $\mathbf{F}$ is a graded field (in the sense of Lemma \ref{lemma:folklore}), then the spectral diameter is finite,
\item\label{principle-b} If there are two disjoint Lagrangians with non-vanishing Lagrangian quantum cohomology (potentially with local systems over some field extension $\mathbf{K}$ of $\mathbf{F}$), then the spectral diameter over $\mathbf{F}$ is infinite.
\end{enumerate}
It is natural to wonder whether these principles are complete. This leads us to ask a question:
\begin{question}
  Is the spectral diameter of $\mathrm{Gr}(k,n)$ infinite, over a field of characteristic $p$, whenever the quantum cohomology is not a graded field. In such cases, can one find disjoint Lagrangians $L_{1},L_{2}$ as in \ref{principle-b}?
\end{question}

See Theorem \ref{theorem:only-applies} for the list of tuples $k,n,p$ not covered by Lemma \ref{lemma:folklore}.

\subsection{Acknowledgements}
\label{sec:acknowledgements}

The first three named authors benefitted from invaluable discussions with Egor Shelukhin and Octav Cornea. Additionally, the first named author wishes to thank Yash Deshmukh and Jun Zhang for useful discussions; the second named author wishes to thank Cheuk Yu Mak for helpful comments; the third named author wishes to thank Filip Bro\'ci\'c, Adi Dickstein, Mark Gudiev, and David Keren Yaar for helpful discussions. The fourth named author wishes to thank his Ph.D.\@ advisors Frédéric Bourgeois and Claude Viterbo for providing guidance and support.

\subsubsection*{Funding}
\label{sec:funding}

This research was partially supported by a Royal Society University Research Fellowship, the ANR grant CoSy, the laboratoire mathématique d'Orsay, the USTC-IGP, and the Université de Montréal DMS.

\section{Proofs of theorems and lemmas}
\label{sec:proof-theorems-and-lemmas}

\subsection{Proof of Lemma \ref{lemma:folklore}}
\label{sec:proof-lemma-folklore}
We will sketch the proof rather quickly, as it uses well-established ideas. The main technical tool in the proof is the \emph{Poincaré duality} formula for spectral invariants originally due to \cite{entov-poltero-IMRN-2003};\footnote{\cite{entov-poltero-IMRN-2003} treat the case $\mathbf{F}=\mathbf{C}$. See, e.g., \cite[\S1.4]{kawamoto-shelukhin-AIM-2024} for general coefficient fields.} it is used to relate the spectral invariants of $H_{t}$ to those of the reversed system $\bar{H}_{t}$:
\begin{lemma}[{\cite[Lemma 2.2.]{entov-poltero-IMRN-2003}}]\label{lemma:poincaré-duality}
  If $a\in \mathrm{QH}^{0}(M,\mathbf{F})$ then:
  \begin{equation*}
    c(a,\bar{H}_{t})=- \sup \{ c(b, H_{t}) : a\ast b=\mathrm{PD}(\mathrm{pt})+qA_{1}+q^{2}A_{2}+\dots \},
  \end{equation*}
  where the supremum is over classes $b\in \mathrm{QH}^{2n}(M,\mathbf{F})$ satisfying the stated condition; here $A_{i}\in H^{*}(M,\mathbf{F})$.\hfill$\square$
\end{lemma}
We refer the reader to \cite{ostrover-qm-non-monotone-2006} for related discussion. As a result:
\begin{equation*}
  \gamma(H_{t};\mathbf{F})=\sup \set{c(b,H_{t})-c(1,H_{t}):b=\mathrm{PD}(\mathrm{pt})+qA_{1}+q^{2}A_{2}+\dots}.
\end{equation*}
More generally, if $e$ is any degree zero idempotent element, one has:
\begin{equation*}
  \gamma_{e}(H_{t};\mathbf{F})=\sup \set{c(b,H_{t})-c(1,H_{t}):b\ast e=\mathrm{PD}(\mathrm{pt})+qA_{1}+q^{2}A_{2}+\dots},
\end{equation*}
where $\gamma_{e}(H_{t};\mathbf{F})=-c(e,H_{t})-c(e,\bar{H}_{t})$.

Let $\dim M = 2n$. Suppose that $e$ satisfies: for $b\in \mathrm{QH}^{2n}(M,\mathbf{F})$, if $e\ast b\ne 0$ then there is $g\in \mathrm{QH}^{-2n}(M,\mathbf{F})$ such that $g\ast b\ast e=e$ (we say that $e$ generates a field factor). Then the \emph{triangle inequality}\footnote{See, e.g., \cite{oh-2005-duke}; our signs are reversed due to our use of cohomological conventions.} for spectral invariants yields:
\begin{equation*}
  c(b,H_{t})+c(g\ast e,0)\le c(e,H_{t}).
\end{equation*}
Since $e\ast g$ has known degree $-2n$, $c(e\ast g,0)\ge -\lfloor{2n/N}\rfloor$ where $N$ is the minimal Chern number --- indeed, one only needs to show the most negative power of $q$ potentially appearing in $e\ast g$ is $q^{-\lfloor 2n/N\rfloor}$. Thus:
\begin{equation*}
  \gamma_{e}(H_{t},\mathbf{F})\le \lfloor{2n/N}\rfloor.
\end{equation*}
The desired result follows.\hfill$\square$

\begin{remark}
  Inspection of the proof shows only the degree $2n$ elements: $$\mathrm{PD}(\mathrm{pt})+qA_{1}+q^{2}A_{2}+\dots$$ are required to be invertible (in the case $e=1$), while Lemma \ref{lemma:folklore} requires that \emph{every} homogeneous element must be invertible. In fact, these are the same: if all elements of the above form are invertible, then, by non-degeneracy of the classical cup-product, all non-zero homogeneous elements are invertible.
\end{remark}

\subsection{Proof of Lemma \ref{lemma:coefficient-extension}}
\label{sec:proof-lemma-coefficient-extension}

See \cite[Proposition 3.1.5]{kawamoto-shelukhin-AIM-2024} for a related statement (involving coefficient rings which are not fields).

Let us recall that $\mathrm{CF}(H_{t};\mathbf{F})$ is the $\mathbf{F}$-vector space spanned by capped orbits of the system $H_{t}$; the spectral invariant of a class $a\in \mathrm{QH}^{*}(M,\mathbf{F})$ is the max-min applied to the class $\mathrm{PSS}(a)\in \mathrm{HF}(H_{t};\mathbf{F})$; this is reviewed in \S\ref{sec:floer-cohomology-review}.

If $\mathbf{F}\to \mathbf{K}$ is a field extension, then there is an inclusion of subcomplexes: $$\mathrm{CF}(H_{t};\mathbf{F})\subset \mathrm{CF}(H_{t};\mathbf{K}).$$ Clearly, the spectral invariant of the unit will not decrease when computed in the larger complex (because we maximize over representatives, the old representatives from $\mathrm{CF}(H_{t};\mathbf{F})$ are still valid). Thus it remains only to show that the spectral invariant also does not increase. For this, we appeal to Lemma \ref{lemma:poincaré-duality}: it is sufficient to show that:
\begin{equation*}
  \sup\set{c(b,\bar{H}_{t}):b=\mathrm{PD}(\mathrm{pt})+qA_{1}+q^{2}A_{2}+\dots}
\end{equation*}
does not decrease. However, we know that for classes $b$ in $\mathrm{CF}(H_{t};\mathbf{F})$ their spectral invariant cannot decrease (and adding new possibilities for $b$ in $\mathrm{CF}(H_{t};\mathbf{K})$ can only increase the supremum). This completes the proof.\hfill$\square$

\subsection{Proof of Lemma \ref{lemma:when-is-field}}
\label{sec:proof-of-lemma-when-is-field}

This lemma is specific to the quantum cohomology of $M=\mathrm{Gr}(k,n)$. In fact, the result will hold in any monotone symplectic manifold satisfying the property that $\mathrm{PD}(\mathrm{pt})$, the generator of the cohomology group with top degree $\dim M$, is invertible when considered as an element in $\mathrm{QH}^{\dim M}(M,\mathbf{F})$. This property is indeed satisfied for $M=\mathrm{Gr}(k,n)$ (in \S\ref{sec:quantum-pieri-rule} we show $\mathrm{PD}(\mathrm{pt})=x_{k}^{n-k}$ and $x_{k}^{n}=q^{k}$, so $\mathrm{PD}(\mathrm{pt})$ is invertible).

Recall the statement of Lemma \ref{lemma:when-is-field}. We suppose that $\mathrm{QH}^0(M,\mathbf{F})$ is a field and let $a\in\mathrm{QH}^*(M,\mathbf{F})$ be a non-zero homogeneous element; we require proving that $a$ is invertible. Without loss of generality, suppose that:
\begin{equation*}
  a=a_{0}+qa_{1}+q^{2}a_{2}+\dots,
\end{equation*}
where $a_{0}\ne 0$ is in $H^{r}(M,\mathbf{F})$. By non-degeneracy of the classical cup product, there is some $b\in H^{\dim M-r}(M,\mathbf{F})$ such that $b\smile a_{0}=\mathrm{PD}(\mathrm{pt})$. It then follows that
\begin{equation*}
  b\ast a=\mathrm{PD}(\mathrm{pt})+qA_{1}+q^{2}A_{2}+\cdots,
\end{equation*}
where $A_{i}\in H^{*}(M)$ are homogeneous elements. Then $\mathrm{PD}(\mathrm{pt})^{-1}b\ast a$ is a non-zero element in $\mathrm{QH}^{0}$, and so is invertible; the desired result follows. $\hfill\square$

\subsection{Proof of Theorem \ref{theorem:technical-main}}
\label{sec:proof-theorem-tech-main}

Denote by $v_{j}=\sigma_{D_{j}}$ where $D_{j}$ has $j$ boxes in the second row and $2\ell-1-j$ boxes in the top row. These elements $v_{0},\dots,v_{\ell-1}$ form a basis for $\mathrm{QH}^{2\ell-1}(\mathrm{Gr}(2,2\ell+1))$. The goal is to determine by which matrix $M$ a particular degree zero element $A$ acts on this graded piece.

\subsubsection{Computation of the matrix}
\label{sec:claim-matrix-identity}

We first claim the following identity:
\begin{equation*}
  (1-q^{-1}x_{2}\ast v_{1})\ast v_{j}=\sum M_{ij}v_{i},
\end{equation*}
where $M_{ij}$ is the matrix in the statement of Theorem \ref{theorem:technical-main}. This degree zero element $A=1-q^{-1}x_{2}\ast v_{1}$ will satisfy the first part of Theorem \ref{theorem:technical-main}.
\begin{proof}[Proof of the claim]
  Let us denote by $V_{a,b}=\sigma_{D_{a,b}}$ where $D$ has $a$ boxes in the top row and $b$ boxes in the second row. Then:
  \begin{equation*}
    x_{2}=V_{1,1}\text{ and }v_{1}=V_{2\ell-2,1}.
  \end{equation*}
  Thus we compute, by straightforward application of the transposed quantum Pieri rule:\footnote{The \emph{transposed quantum Pieri rule} governs how one multiplies by the classes $V_{j,0}$, and is just the rule one obtains from the transposition isomorphism: $$\mathrm{QH}^{*}(\mathrm{Gr}(2,2\ell+1))\simeq \mathrm{QH}^{*}(\mathrm{Gr}(2\ell-1,2\ell+1))$$ and the quantum Pieri rule of \S\ref{sec:quantum-pieri-rule}.}
  \begin{equation*}
    \begin{aligned}
      V_{2\ell-3,0}\ast v_{0}&=V_{2\ell-1,2\ell-3}\\
      V_{2\ell-3,0}\ast v_{1}&=V_{2\ell-1,2\ell-3}+V_{2\ell-2,2\ell-2}+q(0+0+V_{2\ell-5,0})\\
      V_{2\ell-3,0}\ast v_{2}&=V_{2\ell-1,2\ell-3}+q(0+V_{2\ell-5,0}+V_{2\ell-6,1})\\
      V_{2\ell-3,0}\ast v_{3}&=q(V_{2\ell-5,0}+V_{2\ell-6,1}+V_{2\ell-7,2})\\
      V_{2\ell-3,0}\ast v_{3}&=q(V_{2\ell-6,1}+V_{2\ell-7,2}+V_{2\ell-8,3})\\
      \dots&=\dots\\
      V_{2\ell-3,0}\ast v_{\ell-2}&=q(V_{\ell,\ell-5}+V_{\ell-1,\ell-4}+V_{\ell-2,\ell-3})\\
      V_{2\ell-3,0}\ast v_{\ell-1}&=q(V_{\ell-1,\ell-4}+V_{\ell-2,\ell-3}+0).
    \end{aligned}
  \end{equation*}
  Using $v_{1}=x_{2}\ast V_{2\ell-3,0}$, with $x_{2}=V_{1,1}$, the quantum Pieri rule of \S\ref{sec:quantum-pieri-rule} gives:
  \begin{equation*}
    \begin{aligned}
      (x_{2}\ast v_{1})\ast v_{0}&=q(0+0+V_{2\ell-2,1})=qv_{1}\\
      (x_{2}\ast v_{1})\ast v_{1}&=q(V_{2\ell-1,0}+V_{2\ell-2,1}+V_{2\ell-3,2})=q(v_{0}+v_{1}+v_{2})\\
      (x_{2}\ast v_{1})\ast v_{2}&=q(V_{2\ell-2,1}+V_{2\ell-3,2}+V_{2\ell-4,3})=q(v_{1}+v_{2}+v_{3})\\
      (x_{2}\ast v_{1})\ast v_{3}&=q(V_{2\ell-3,2}+V_{2\ell-4,3}+V_{2\ell-5,4})=q(v_{2}+v_{3}+v_{4})\\
      \dots&=\dots\\
      (x_{2}\ast v_{1})\ast v_{\ell-2}&=q(V_{\ell+2,\ell-3}+V_{\ell+1,\ell-2}+V_{\ell,\ell-1})=q(v_{\ell-3}+v_{\ell-2}+v_{\ell-1})\\
      (x_{2}\ast v_{1})\ast v_{\ell-1}&=q(V_{\ell+1,\ell-2}+V_{\ell,\ell-1}+0)=q(v_{\ell-2}+v_{\ell-1}).
    \end{aligned}
  \end{equation*}
  Thus $1-q^{-1}x_{2}\ast v_{1}$ acts according to the matrix $M$ from the statement of Theorem \ref{theorem:technical-main}, as desired.
\end{proof}

\subsubsection{The characteristic polynomial of the matrix $M$}
\label{sec:char-polyn-matr}

Let $\pi_{\ell}(x)$ be the characteristic polynomial of the matrix $M$, i.e., $\pi_{\ell}(x)=\det(M-xI)$. Then:
\begin{equation*}
  R_{\ell}(x)=x^{\ell}\pi_{\ell}(-x-x^{-1})=\frac{x^{n}-1}{x-1}=x^{n-1}+x^{n-2}+\dots+x+1,
\end{equation*}
where $n=2\ell+1$.
\begin{proof}[Proof of the claim]
  We observe that:
  \begin{equation*}
    R_{\ell}(x)=\left|
      \begin{smallmatrix}
        x^{2}+x+1&-x&0&0&0&0\\
        -x&x^{2}+1&-x&0&0&0\\
        0&-x&x^{2}+1&-x&0&0\\
        0&0&-x&x^{2}+1&-x&0\\
        0&0&0&-x&x^{2}+1&-x\\
        0&0&0&0&-x&x^{2}+1\\
      \end{smallmatrix}
    \right|.
  \end{equation*}
  The determinant of this triagonal matrix can be computed using a recursion relation. One sees that:
  \begin{equation*}
    R_{\ell}(x)=(x^{2}+1)R_{\ell-1}(x)-x^{2}R_{\ell-2}(x),
  \end{equation*}
  with initial conditions $R_{0}(x)=1$ and $R_{1}(x)=1+x+x^{2}$. It is clear by inspection that:
  \begin{equation*}
    R_{\ell}(x)=x^{2\ell}+x^{2\ell-1}+\dots+x+1
  \end{equation*}
  solve this recursion relation; equivalently, one can solve the recursion by diagonalizing the matrix appearing in the recurrence relation:
  \begin{equation*}
    \left[\begin{matrix}
      R_{\ell}(x)\\
      R_{\ell-1}(x)
    \end{matrix}\right]=\left[
    \begin{matrix}
      x^{2}+1&-x^{2}\\
      1&0
    \end{matrix}\right]\left[
    \begin{matrix}
      R_{\ell-1}(x)\\R_{\ell-2}(x)
    \end{matrix}
  \right].
\end{equation*}
This completes the proof.
\end{proof}

\subsubsection{Completing the proof of Theorem \ref{theorem:technical-main}}
\label{sec:compl-proof-theor}

It remains only to show the equivalence of the statements \ref{item:technical-main-1} through \ref{item:technical-main-3} from the statement. Let us abbreviate $\mathrm{QH}^{d}=\mathrm{QH}^{d}(\mathrm{Gr}(2,2\ell+1))$. The fact that \ref{item:technical-main-1} follows from \ref{item:technical-main-2} is basic representation theory: if $\pi(y)$ is irreducible, then $\pi(y)$ is the minimal polynomial of $A$. It follows that the action of $A$ on $\mathrm{QH}^{2\ell-1}$ has no proper non-zero invariant subspaces. We conclude that the action of any non-zero element $B\in \mathrm{QH}^{0}$ as an endomorphism of $\mathrm{QH}^{2\ell-1}$ has zero kernel (otherwise this kernel would be an invariant subspace of the action of $A$). Here we use that $\mathrm{QH}^{2\ell-1}$ has a unit, namely $v_{0}$, to ensure the action is faithful. In particular, $B$ has no zero divisors $C\in \mathrm{QH}^{0}$, otherwise $B\ast C\ast v_{0}=0$, which would contradict the fact that $C\ast v_{0}$ is non-zero element in $\mathrm{QH}^{2\ell-1}$. Thus $\mathrm{QH}^{0}$ is a field, as desired.

That \ref{item:technical-main-1} implies \ref{item:technical-main-2} follows similar lines: if the action of $A$ has any non-trivial proper invariant subspace, then this subspace is invariant for the entire action of $\mathrm{QH}^{0}$, and by dimension reasons one concludes the existence of zero divisors.\footnote{Note that $\dim \mathrm{QH}^{0}=1+\dim H^{n}=\dim H^{n-2}=\dim \mathrm{QH}^{2\ell-1}$, recalling $n-2=2\ell-1$.}

Next we prove that \ref{item:technical-main-2} implies \ref{item:technical-main-3}, using a fundamental idea of Galois theory; let $\mathbf{L}$ be the splitting field of $\pi(y)$ and suppose that the splitting of $\pi(y)$ has no repeated factors in $\mathbf{L}$, then: {\itshape the group of automorphisms of $\mathbf{F}\subset \mathbf{L}$ acts transitively on the roots of $\pi(y)$ if and only if $\pi(y)$ is irreducible.} This uses \cite[Theorem VII.6.9]{aluffi-book-2009} which asserts that $\mathbf{F}$ is the ``fixed field'' of $\mathrm{Aut}_{\mathbf{F}}(\mathbf{L})$; in particular, if the group of automorphisms did not act transitively, then one would conclude $\pi(y)$ has a factor over $\mathbf{F}$. 

Let us observe that, over $\mathbf{L}$, we can write:
\begin{equation*}
  \pi(y)=(-1)^{\ell}(y-r_{1})\dots (y-r_{\ell})
\end{equation*}
and hence:
\begin{equation*}
  x^{\ell}\pi(-x-x^{-1})=(x^{2}+r_{1}x+1)\dots (x^{2}+r_{\ell}x+1)=\frac{x^{n}-1}{x-1}.
\end{equation*}
It follows that each $r_{i}$ lies in the splitting field of $x^{n}-1$, denoted $\mathbf{F}(\zeta)$; i.e., $\mathbf{L}\subset \mathbf{F}(\zeta)$. In fact there is some ordering of roots so that:
\begin{equation*}
  r_{i}=-s_{i}-s_{-i}\text{ and }s_{-i}=s_{i}^{-1}\text{ for }i=1,\dots,\ell
\end{equation*}
where $x^{n-1}+\dots+x+1=(x-s_{1})(x-s_{-1})\dots(x-s_{\ell})(x-s_{-\ell})$ over $\mathbf{F}(\zeta)$.

Because we are working with $\mathbf{F}=\Z/p\Z$, the mod $p$ Frobenius automorphism generates the automorphism groups of field extensions of $\mathbf{F}$; see \cite[Proposition 5.8]{aluffi-book-2009}. In particular, if $\pi(y)$ is irreducible, it follows that we can reorder the $\ell$ roots of $\pi$ (which are distinct elements in $\mathbf{L}$) so that:
\begin{equation}\label{eq:acting-transitively}
  r_{i}=r_{i-1}^{p}\text{ and so }s_{i}+s_{i}^{-1}=(s_{i-1}+s_{i-1}^{-1})^{p}.
\end{equation}
We are almost done with the implication \ref{item:technical-main-2}${}\implies{}$\ref{item:technical-main-3}. It remains only to prove that $n$ is prime, coprime to $p$, and $\set{-1,p}$ generate the units $(\Z/n\Z)^{\times}$. If $n$ were not prime then $q|n$ for some prime number $q$. The $q$ root of unity $\xi$ is technically an $n$th root of unity, and hence $\xi=s_{i}$ for some $i$. It follows that $\xi=s_{i}$ is in $\mathbf{F}(\xi)$. But then, since the Frobenius acts transitively on the roots $r_{1},\dots,r_{\ell}$, it would follow that $r_{1},\dots,r_{\ell}$ all lie in $\mathbf{F}(\xi)$, and hence $\mathbf{L}\subset \mathbf{F}(\xi)$. But this contradicts:
\begin{itemize}
\item $[\mathbf{L}:\mathbf{F}]=\ell$,
\item $[\mathbf{F}(\xi):\mathbf{F}]=q-1$ (or $1$ if $q=p$, technically speaking);
\end{itemize}
this is a contradiction, because we may assume that $n\ge 9$ is odd and appeal to the (more-or-less obvious) inequality $q-1<n/3-1<(n-1)/2=\ell,$ and conclude $\mathbf{L}$ cannot lie in $\mathbf{F}(\xi)$.

There is one possibility we have not covered, namely the case $n=p$, which must be excluded. Here $x^{n}-1=(x-1)^{p}$ and so it follows in fact that all roots $r_{1},\dots,r_{\ell}$ lie already in $\mathbf{F}$, contradicting irreducibility of $\pi$.

Finally, we show that $\set{-1,p}$ generate the units in $(\Z/n\Z)^{\times}$. Since the powers:
\begin{equation*}
  -\zeta^{p^{i}}-\zeta^{-p^{i}}\text{ for }i=0,\dots,\ell-1
\end{equation*}
parametrize the $\ell$ roots of $\pi(y)$, it follows that every element in $(\Z/n\Z)^{\times}$ is of the form $p^{i}$ or $-p^{i}$, as desired. The implication \ref{item:technical-main-3}${}\implies{}$\ref{item:technical-main-2} follows the reverse argument. The assumptions imply that the Frobenius map acts transitively on the roots of $\pi(y)$, and hence $\pi(y)$ must be irreducible. We should establish that $\pi(y)$ has distinct roots over $\mathbf{L}$ when appealing to \cite[Theorem VII.6.9]{aluffi-book-2009}. It is clear that, if two roots are the same then there are \emph{four} distinct non-trivial solutions of:
\begin{equation*}
  \zeta+\zeta^{-1}=r
\end{equation*}
for some $r$. But this contradicts the fact that a quadratic polynomial has at most two roots. This completes the proof of Theorem \ref{theorem:technical-main}.\hfill$\square$

\subsubsection{Proof of Theorem \ref{theorem:main}}
\label{sec:proof-theorem-main}

The proof of Theorem \ref{theorem:main} follows by an immediate combination of Theorem \ref{theorem:technical-main}, Lemma \ref{lemma:when-is-field} and Lemma \ref{lemma:folklore}.\hfill$\square$

\subsubsection{Proof of Theorem \ref{theorem:main-over-Q}}
\label{sec:proof-theorem-main-over-Q}

The proof of Theorem \ref{theorem:main} follows the same argument as above; we only need to explain the ``specialization to primes'' argument ensuring that $\mathrm{QH}^{0}(\mathbf{Q}):=\mathrm{QH}^{0}(\mathrm{Gr}(2,n),\Q)$ is a field, when $n$ is a prime number. Given $n$, there exists \emph{some} prime number $p$ such that $\set{-1,p}$ generates the units $(\Z/n\Z)^{\times}$.

Let us denote by $\mathbf{F}=\Z/p\Z$. If $\mathrm{QH}^{0}(\Q)$ is not a field, then there exists a zero divisor $a$. Multiplication by $a$ against some basis is represented by some matrix with rational coefficients (using a basis of Schubert classes). Then $Na$ can be considered as a class in $\mathrm{QH}^{0}(\mathbf{F})$ for large $N$ (clearing denominators); however, since $\mathrm{QH}^{0}(\mathbf{F})$ is known to be a field, we must have that $Na$ is zero. By picking $N$ intelligently (so that at least one coefficient of $Na$ is nonzero mod $p$), we obtain a contradiction.\hfill$\square$
           
\subsection{Proof of Theorem \ref{theorem:matrix-when-n-even}}
\label{sec:proof-theorem-matrix-n-even}

In a manner similar to \S\ref{sec:proof-theorem-tech-main}, we compute:
\begin{equation*}
  \begin{aligned}
    v_{1}\ast v_{0}&=q(0+0+V_{2\ell-2,0})\\
    v_{1}\ast v_{1}&=q(V_{2\ell,2\ell}+V_{2\ell-2,0}+V_{2\ell-3,1})\\
    v_{1}\ast v_{i}&=q(V_{2\ell-i,i-2}+V_{2\ell-i-1,i-1}+V_{2\ell-i-2,i})\text{ for }2\le i<\ell\\
    v_{1}\ast v_{\ell}&=q(V_{\ell,\ell-2}+0+0).
  \end{aligned}
\end{equation*}
Multiplying by $x_{2}$ then yields:
\begin{equation*}
  \begin{aligned}
    (x_{2}\ast v_{1})\ast v_{0}&=q(0+0+v_{1})\\
    (x_{2}\ast v_{1})\ast v_{1}&=q(v_{0}+v_{1}+v_{2})\\
    (x_{2}\ast v_{1})\ast v_{i}&=q(v_{i-1}+v_{i}+v_{i+1})\text{ for }2\le i<\ell\\
    (x_{2}\ast v_{1})\ast v_{\ell}&=q(v_{\ell-1}+0+0).
  \end{aligned}
\end{equation*}
Thus $A=1-q^{-1}x_{2}\ast v_{1}$ satisfies the first part of Theorem \ref{theorem:matrix-when-n-even}. It remains to prove the statement about the characteristic polynomial $\pi(y)=\mathrm{det}(M-yI)$. We have:
\begin{equation*}
  x^{\ell+1}\pi(-x-x^{-1})=\left|\begin{smallmatrix}
  x^{2}+x+1&-x&0&0&0&0\\
  -x&x^{2}+1&-x&0&0&0\\
  0&-x&x^{2}+1&-x&0&0\\
  0&0&-x&x^{2}+1&-x&0\\
  0&0&0&-x&x^{2}+1&-x\\
  0&0&0&0&-x&x^{2}+x+1\\
  \end{smallmatrix}
  \right|.
\end{equation*}
Using linearity of the determinant as function of the last row, we conclude:
\begin{equation*}
  x^{\ell+1}\pi(-x-x^{-1})=R_{\ell+1}(x)+xR_{\ell}(x)
\end{equation*}
where $R_{\ell}(x)$ is as in \S\ref{sec:char-polyn-matr}. Thus, the results of \S\ref{sec:char-polyn-matr} yield:
\begin{equation*}
  x^{\ell+1}\pi(-x-x^{-1})=(1+x)(x^{n-1}+x^{n-2}+\dots+x+1),
\end{equation*}
as desired.\hfill$\square$

\subsection{Decomposition of $\pi(y)$ into irreducible factors}
\label{sec:decomp-piy-into}

Theorems \ref{theorem:technical-main} and \ref{theorem:matrix-when-n-even} give sufficient information to answer the following question: \emph{how many irreducible factors does the characteristic polynomial of $\pi(y)$ split into}. This is relevant to studying the semisimplicity of $\mathrm{QH}^{0}(\mathrm{Gr}(2,n),\mathbf{F})$. It will be necessary to assume that $p$ is coprime with $n$.

Let us continue with the discussion from \S\ref{sec:proof-theorem-matrix-n-even} where $n$ was even. Over the algebraic closure of $\mathbf{F}$, we can write $x^{\ell+1}\pi(-x-x^{-1})$ as:
\begin{equation*}
 (x^{2}+r_{0}x+1)\dots(x^{2}+r_{\ell}x+1)=(x+1)^{2}\prod_{i=1}^{\ell}(x-s_{i})(x-s_{-i}),
\end{equation*}
and after relabelling, we can assume that each $r_{0}=2$ and $-s_{i}-s_{-i}=r_{i}$ and $s_{-i}=s_{i}^{-1}$. In particular, the roots of $\pi(y)$ lie in the splitting field $\mathbf{F}(\zeta)$ of $x^{n-1}+\dots+x+1$, and are precisely the sums:
\begin{equation*}
  r=-s-s^{-1}
\end{equation*}
where $s$ is a root of $x^{n-1}+\dots+x+1$.

We have obtained exactly the analogous result when $n$ is odd: {\itshape the negative of each such sum $r=-s+s^{-1}$ is a root of $\pi(y)$, where $\pi(y)$ is the characteristic polynomial of the matrix determined by multiplication by $A=q^{-1}x_{2}\ast v_{1}$ on the graded piece space $\mathrm{QH}^{n-2}(\mathrm{Gr}(2,n),\mathbf{F})$.} Henceforth we allow $n$ to be even or odd in the discussion.

The Frobenius automorphism preserves this set of roots, and we conclude the following general statement for Grassmannians of the form $\mathrm{Gr}(2,n)$:
\begin{prop}\label{proposition:secret-semisimple}
  If $\mathbf{F}=\Z/p\Z$ and $p$ is coprime to $n$, and there are $N$ orbits of the action $s+s^{-1}\mapsto s^{p}+s^{-p}$ on the set of sums $s+s^{-1}$ where $s$ is a root of the polynomial $x^{n-1}+\dots+x+1$, then $\mathrm{QH}^{0}(\mathrm{Gr}(2,n),\mathbf{F})$ splits into a direct sum of $N$ fields, and $\mathrm{QH}^{0}(\mathrm{Gr}(2,n),\mathbf{Q})$ splits into at most $N$ fields.
\end{prop}
\begin{proof}
   The (standard) idea is to average the roots:
  \begin{equation*}
    f_{O}(y)=\prod_{r\in O} (y-r)
  \end{equation*}
  where $O$ is an orbit of the aforementioned action. These $f_{O}(y)$ form the factors of $\pi(y)$:
  \begin{equation*}
    \pi(y)=\prod_{\text{orbits }O}f_{O}(y)^{m_{O}}.
  \end{equation*}
  We will use the condition that $n$ is coprime with $p$ to prove the multiplicities $m_{O}$ are $1$, for each $O$; then standard representation theory, as in \S\ref{sec:compl-proof-theor}, implies that $\mathrm{QH}^{0}(\mathrm{Gr}(2,n),\mathbf{F})$ decomposes into $N$ many fields (in this case, the minimal polynomial of the element $A$ equals its characteristic polynomial).\footnote{Suppose that $\pi(y)=f_{1}(y)\dots f_{N}(y)$; then one can solve:
    \begin{equation*}
      1=\sum_{i=1}^{N}g_{i}(x)\prod_{j\ne i}f_{j}(x).
    \end{equation*}
    This determines how $\mathrm{QH}^{0}$ splits into fields; if $A=q^{-1}x^{2}(v_{0}-v_{1})$ then the elements:
    \begin{equation*}
      E_{i}=g_{i}(A)\prod_{j\ne i}g_{j}(A)\text{ where }i=1,\dots,N
    \end{equation*}
    are non-zero idempotents satisfying $E_{i}E_{j}=0$ for $i\ne j$. Each $E_{i}$ generates a field factor.} If $m_{O}$ were non-zero, then $\pi(y)$ would have a repeated root $r$, and then $(x^{2}+rx+1)^{2}$ would divide $x^{n}-1$, contradicting the fact that $x^{n}-1$ has no repeated irreducible factors when $n$ is coprime to $p$.

  The last part of the statement follows from a specialization to primes argument, similarly \S\ref{sec:proof-theorem-main-over-Q}.
\end{proof}

\begin{remark}
When $n$ is coprime with $p$, there is a primitive root of unity $s=\zeta$, and then the sums can be identified with pairs $\set{a,b}\in \Z/n\Z$ where $a+b=0$ and where $a\ne b\ne 0$, modulo $n$. The action is $p\set{a,b}=\set{pa,pb}$. For instance, if $n=10$, then over $\mathbf{F}=\Z/7\Z$ there are three orbits:
\begin{itemize}
\item $\set{1,9}\mapsto \set{3,7}\mapsto \set{1,9}$,
\item $\set{2,8}\mapsto \set{6,4}\mapsto \set{2,8}$,
\item $\set{5,5}\mapsto \set{5,5}$,
\end{itemize}
and so $\mathrm{QH}^{0}(\mathrm{Gr}(2,10),\mathbf{F})\simeq \mathbf{F}\oplus \mathbf{K}\oplus \mathbf{K}$ where $[\mathbf{K}:\mathbf{F}]=2$.
\end{remark}

\begin{remark}
  Interestingly enough, if the assumption that $n$ is coprime to $p$ is dropped, then $\mathrm{QH}^{0}(\mathrm{Gr}(2,n),\mathbf{F})$ may no longer admit field summands.
\end{remark}

\subsection{Proof of Theorem \ref{theorem:ring-homomorphism}}
\label{sec:proof-theorem-ring-homomorphism}

In this section we explain how to construct ring homomorphisms from $\mathrm{QH}^{*}(\mathrm{Gr}(k,n),\mathbf{F})$ into a field extension $\mathbf{K}$ of $\mathbf{F}$.

\subsubsection{Symmetric polynomials}
\label{sec:symm-polyn}

To study $\mathrm{QH}(\mathrm{Gr}(k,n);\mathbf{F})$, one important idea is to consider the algebra homomorphism:
\begin{equation}\label{eq:ansatz}
  \mathrm{QH}(\mathrm{Gr}(k,n);\mathbf{F})\mapsto \mathbf{F}[q^{-1},q,z_{1},\dots,z_{k}]/\mathscr{J}
\end{equation}
where the Chern class $x_{i}$ is sent to the $i$th elementary symmetric polynomial:
\begin{equation*}
  e_{i}(z_{1},\dots,z_{k})=\sum z_{1}^{i_{1}}\dots z_{k}^{i_{k}};
\end{equation*}
the sum is over all $i=i_{1}+\dots+i_{k}$ such that each $i_{j}$ is $0$ or $1$.
\begin{claim}
  The map is well-defined provided that:
  \begin{equation*}
    \mathscr{J}=\text{ideal generated by }h_{n-k+1},\dots,h_{n-1},h_{n}+(-1)^{k}q,
  \end{equation*}
  where $h_{i}$ is the $i$th complete symmetric polynomial:
  \begin{equation*}
    h_{i}(z_{1},\dots,z_{k})=\sum_{i_{1}+\dots+i_{k}=i} z_{1}^{i_{1}}\dots z_{k}^{i_{k}}.
  \end{equation*}
\end{claim}
\begin{proof}
  This is used in \cite{galkin-golyshev-RMS-2006} and attributed to \cite{siebert-tian-AJM-1997}. One can also derive from the result of \cite{buch-compositio-2003} as follows. The key is to use the ``generating'' functions:
  \begin{equation*}
    E(t)=\sum e_{i}t^{i}\text{ and }H(t)=\sum h_{i}t^{i}
  \end{equation*}
  valued in $\mathbf{F}[t,x_{1},\dots,x_{k}]$. These satisfy the identity:\footnote{See \cite[pp.\,20]{macdonald-oxford-1995}.}
  \begin{equation*}
    E(-t)H(t)=1.
  \end{equation*}
  In particular:
  \begin{equation*}
    h_{r}=h_{r-1}e_{1}-h_{r-2}e_{2}+\dots-(-1)^{r}h_{0}e_{r}=0.
  \end{equation*}
  Let us note that $Y_{r}$, given in \eqref{eq:recursion} satisfies a similar recursion relation; indeed, using the Laplace rule for computing determinants, we obtain:
  \begin{equation*}
    Y_{4}=\left|
      \begin{smallmatrix}
        x_{1}&x_{2}&x_{3}&x_{4}\\
        x_{0}&x_{1}&x_{2}&x_{3}\\
        0&x_{0}&x_{1}&x_{2}\\
        0&0&x_{0}&x_{1}
      \end{smallmatrix}
    \right|=x_{1}\left|
      \begin{smallmatrix}
        x_{1}&x_{2}&x_{3}\\
        x_{0}&x_{1}&x_{2}\\
        0&x_{0}&x_{1}
      \end{smallmatrix}
    \right|-x_{2}\left|
      \begin{smallmatrix}
        x_{0}&x_{2}&x_{3}\\
        0&x_{1}&x_{2}\\
        0&x_{0}&x_{1}
      \end{smallmatrix}\right|+x_{3}\left|
      \begin{smallmatrix}
        x_{0}&x_{1}&x_{3}\\
        0&x_{0}&x_{2}\\
        0&0&x_{1}
      \end{smallmatrix}\right|-x_{4}\left|
      \begin{smallmatrix}
        x_{0}&x_{1}&x_{3}\\
        0&x_{0}&x_{2}\\
        0&0&x_{0}
      \end{smallmatrix}\right|
  \end{equation*}
  which simplifies to $Y_{4}=x_{1}Y_{3}-x_{2}Y_{2}+x_{3}Y_{1}-x_{4}$. This holds in general, as can be easily checked by the reader. Hence the map:
  \begin{equation*}
    \mathbf{F}[q^{-1},q,x_{1},\dots,x_{k}]\to \mathbf{F}[q^{-1},q,z_{1},\dots,z_{k}]
  \end{equation*}
  sending $x_{i}$ to $e_{i}(z)$ (and $q$ to $q$) necessarily sends $Y_{i}$ to $h_{i}(z)$. Thus, \eqref{eq:ansatz} is well-defined provided we quotient by $\mathscr{J}$ from the statement.
\end{proof}

In fact, let us denote by $\Lambda\subset \mathbf{F}[z_{1},\dots,z_{k}]$ the ring of symmetric polynomials. Then $\Lambda$ is generated by the elementary polynomials $e_{1},\dots,e_{k}$ (see \cite[\S VII.7.3]{aluffi-book-2009}), and we have constructed an isomorphism:
\begin{equation}\label{eq:JQ-isomorphism}
  \mathrm{QH}^{*}(\mathrm{Gr}(k,n);\mathbf{F})\mapsto \Lambda[q^{-1},q]/\mathscr{J},
\end{equation}
where $\mathscr{J}$ is as in the claim.

\subsubsection{Roots of unity}
\label{sec:roots-unity}
As in the statement of Theorem \ref{theorem:ring-homomorphism}, fix $\xi$ satisfying $\xi^{n}+(-1)^{k}=0$, and consider evaluation maps:
\begin{equation}\label{eq:galkin-golyshev-map-1}
  \mathrm{ev}_{J}:\mathrm{QH}^{0}(\mathrm{Gr}(k,n);\mathbf{F})\to \mathbf{K}
\end{equation}
where $\mathbf{K}$ is the splitting field for $x^{n}+(-1)^{k}$. The map $\mathrm{ev}_{J}$ is defined by precomposing a map of the form:
\begin{equation}\label{eq:galkin-golyshev-map-2}
  \mathrm{ev}_{J}:\mathbf{F}[q^{-1},q,z_{1},\dots,z_{k}]\to \mathbf{K},
\end{equation}
with the ring map $\mathrm{QH}^{0}(\mathrm{Gr}(k,n);\mathbf{F})\to \Lambda[q^{-1},q]/\mathscr{J}$ of \eqref{eq:JQ-isomorphism}; the map \eqref{eq:galkin-golyshev-map-2} is chosen such that:
\begin{itemize}
\item each $z_{i}$ is sent to $\xi \zeta_{i}$ where $\zeta_{i}$ is an $n$th root of unity;
\item the collection $\set{\zeta_{1},\dots,\zeta_{k}}$ are admissible (see \S\ref{sec:sketch-proof-theorem-only-applies} for the definition);
\item $q$ is sent to $1$.
\end{itemize}
In the case $p$ is a symmetric polynomial we will write $p(\xi\zeta_{J})$, where $J$ is the choice of roots of unity $\set{\zeta_{1},\dots,\zeta_{k}}$, as the order is irrelevant.

\begin{lemma}
  Evaluation maps \eqref{eq:galkin-golyshev-map-1} are well-defined algebra homomorphisms, i.e., the maps \eqref{eq:galkin-golyshev-map-2} vanish on the generators of $\mathscr{J}$ and the restriction \eqref{eq:galkin-golyshev-map-1} takes values in $\mathbf{F}(\zeta)$, the splitting field of $x^{n}-1$.
\end{lemma}
\begin{proof}
  This is due to \cite{galkin-golyshev-RMS-2006} (at least in the case $\mathrm{char}(\mathbf{F})=0$). For our purposes, we let $\mathrm{char}(\mathbf{F})=p$. The case $\mathrm{char}(\mathbf{F})=0$ is similar to the case when $d=0$ and is left to the reader.

  The idea is to use the generating function identity $E(-t)H(t)=1$ invoked previously, as well as the following application of \emph{Vieta's formula}:
  \begin{equation}
    \sum_{j=0}^{k}e_{j}(\zeta_{J})(-1)^{t}\sum_{i=0}^{n-k}e_{i}(\zeta_{J^{c}})(-t)^{i}=\sum_{i=0}^{n}e_{i}(\zeta_{I})(-t)^{i}=1-t^{n},
  \end{equation}
  where $\zeta_{I}=\set{\zeta_{1},\dots,\zeta_{n}}$ is the complete selection of roots of unity. It then follows from $E(-t)H(t)=1$ that:
  \begin{equation*}
    (1-t^{n})\sum_{j=0}^{\infty} h_{j}(\zeta_{J})t^{j}=\sum_{i=0}^{n-k}e_{i}(\zeta_{J^{c}})(-t)^{i}.
  \end{equation*}
  Consideration of degrees then yields:
  \begin{equation*}
    \sum_{j=1}^{k-1} h_{n-k+j}(\zeta_{J})t^{n-k+j}=0\text{ and }h_{n}(\zeta_{J})-1=0,
  \end{equation*}
  Thus, if we evaluate instead at $\xi \zeta_{J}$, then each term $h_{n-k+j}(\xi \zeta_{J})$ remains zero for $j=1,\dots,k-1$, while $h_{n}(\xi\zeta_{J})-(-\xi)^{n}=h_{n}(\xi\zeta_{J})+(-1)^{k}=0$. The map of \eqref{eq:galkin-golyshev-map-2} is well-defined, as it vanishes on the generators of $\mathscr{J}$.

  Next we prove that the map \eqref{eq:galkin-golyshev-map-1} takes values in $\mathbf{F}(\zeta)$. It is clear that $x_{i}$ is sent into $\xi^{i}\mathbf{F}(\zeta)$, and $q$ is sent into $\xi^{n}\mathbf{F}(\zeta)=\mathbf{F}(\zeta)$. Thus by consideration of the (complex) degree of elements we conclude every element of degree $i$ in $\mathrm{QH}$ is mapped into $\xi^{i}\mathbf{F}(\zeta)$, and so the degree $0$ part is sent into $\mathbf{F}(\zeta)$, as desired. This completes the proof of the lemma, and Theorem \ref{theorem:ring-homomorphism}.
\end{proof}

\subsection{Proof of Theorem \ref{theorem:only-applies}}
\label{sec:proof-theorem-only-applies}

We need to show that, in all cases not covered by Theorem \ref{theorem:main}, the algebra $\mathrm{QH}^{0}(\mathrm{Gr}(k,n))$ is not a field (restricting to the range $n\ge 2k$ to avoid double counting -- we also ignore the cases when $k=2$ and $n$ is odd, as these are already determined by Theorem \ref{theorem:main}).

The proof will follow from linear algebra (dimension estimates) and the existence of the ring homomorphism:
\begin{equation}\label{eq:ring-homomorphism-only-applies}
  \mathrm{QH}^{0}(\mathrm{Gr}(k,n),\mathbf{F})\to \mathbf{F}(\zeta).
\end{equation}
from Theorem \ref{theorem:ring-homomorphism} (as explained in \S\ref{sec:roots-unity}, this is valued in $\mathbf{F}(\zeta)\subset \mathbf{K}$ when we restrict to the degree zero part).

We argue by contradiction: suppose that $\mathrm{QH}^{0}=\mathrm{QH}^{0}(\mathrm{Gr}(k,n),\mathbf{F})$ \emph{is} a field. The first step is to compute the dimension of $\mathrm{QH}^{0}$; by Lemma \ref{lemma:when-is-field}, it follows that $x_{1}$ is invertible (we are assuming that $\mathrm{QH}^{0}$ is a field). Therefore:
\begin{equation*}
  \dim \mathrm{QH}^{0}=\dim \mathrm{QH}^{1}=\dots =\dim\mathrm{QH}^{n-1},
\end{equation*}
as multiplication by $x_{1}$ establishes bijections between these graded pieces. On the other hand:
\begin{equation*}
  \sum_{i=0}^{n-1} \dim \mathrm{QH}^{i}=\sum_{i=0}^{n-1}\sum_{j=-\infty}^{+\infty}\dim H^{i+nj}(\mathrm{Gr}(k,n))=\left(
    \begin{matrix}
      n\\
      k
    \end{matrix}
  \right),
\end{equation*}
where we use that the formal variable $q$ has degree $n$; here, as always, we are using complex dimensions. Thus, assuming that $\mathrm{QH}^{0}$ is a field, we conclude its dimension must be:
\begin{equation*}
  \dim \mathrm{QH}^{0}=\frac{1}{n}\left(
    \begin{matrix}
      n\\
      k
    \end{matrix}
  \right)=\frac{(n-1)(n-2)\dots(n-k+1)}{k!}.
\end{equation*}
Note that, in some cases, this is not even an integer, and so we automatically would conclude that $\mathrm{QH}^{0}$ is not a field. This happens, e.g., if $n$ is even and $k=2$, as it gives $(n-1)/2$. Another option is $k=3$ and $n=6$, as the above gives $20/6$, which is not an integer. However, we can be more precise and exclude all cases not covered by Theorem \ref{theorem:main}.

The key idea is that, if $\mathrm{QH}^{0}$ is a field, then $\dim_{\mathbf{F}} \mathrm{QH}^{0}$ must divide $[\mathbf{F}(\zeta):\mathbf{F}]$, because the ring homomorphism \eqref{eq:ring-homomorphism-only-applies} is a field extension. Thus we conclude (still in search of our contradiction) that:
\begin{equation}\label{eq:division}
  (n-1)(n-2)\dots(n-k+1)\text{ divides }k![\mathbf{F}(\zeta):\mathbf{F}]
\end{equation}
Let us consider the case $n=7$ and $k=3$, in which case we have
\begin{equation*}
  6\times 5\text{ divides }6[\mathbf{F}(\zeta):\mathbf{F}].
\end{equation*}
However, $[\mathbf{F}(\zeta):\mathbf{F}]$ is never divisible by $5$ if $n=7$; indeed, we have either $[\mathbf{F}(\zeta):\mathbf{F}]=6$ if $\mathrm{char}(\mathbf{F})\ne 7$ and $[\mathbf{F}(\zeta):\mathbf{F}]=1$ if $\mathrm{char}(\mathbf{F})=1$.

Now in the remaining cases, we may assume that $n\ge 8$. From \eqref{eq:division}, it follows that:
\begin{equation*}
  \frac{n-1}{[\mathbf{F}(\zeta):\mathbf{F}]}\frac{n-2}{3!}\frac{n-k+1}{k}\dots \frac{n-3}{4}\le 1
\end{equation*}
Since $[\mathbf{F}(\zeta):\mathbf{F}]\le n-1$, and $n-2\ge 3!$, and $n\ge 2k$, it follows that all terms in the product are at least $1$. The last written term is strictly bigger than $1$, and so we get a contradiction if $k\ge 4$ (otherwise the last written term does not actually appear). The final remaining case is when $n=8$ and $k=3$, when we have:
\begin{equation*}
  \frac{7}{[\mathbf{F}(\zeta):\mathbf{F}]}\le 1.
\end{equation*}
This is also a contradiction, since $[\mathbf{F}(\zeta):\mathbf{F}]\le 4$ if $\zeta$ is an 8th root of unity. Thus we have reached a contradiction of our assumption that $\mathrm{QH}^{0}$ is a field in any case not covered already by Theorem \ref{theorem:main}. This completes the proof of Theorem \ref{theorem:only-applies}.\hfill$\square$

\appendix

\section{Floer cohomology review}
\label{sec:floer-cohomology-review}

This appendix reviews the Hamiltonian Floer theory used in this paper. Throughout, we assume that $H: M\times S^{1} \rightarrow \mathbb{R}$ is a non-degenerate Hamiltonian function on a closed monotone symplectic manifold $(M,\omega)$. We denote by $X_t$ the vector-field defined by the relation $\omega(X_t,-)=-\d H_t$.

\subsection{Cohomological cappings}

A representative (cohomological) capping of a loop $x$ is a smooth map $\overline{x}:D\rightarrow M$ whose restriction to $\partial D=S^1$ is given by $x$ when $\bd D$ is oriented as a negative boundary (i.e., clockwise, the opposite of the standard orientation of $\partial D$). The difference of two representative cappings of $x$ form a sphere, and if
this sphere has zero symplectic area then the representatives are deemed \emph{equivalent}. An equivalence class of such representatives will be called a \emph{capping} of $x$. The quotient space of cappings is a covering of the space of contractible loops $\mathcal{L}M$.

\begin{remark}
  The slightly non-standard definition of cappings is motivated by the PSS construction; see Figure \ref{figure:cappings-PSS} and \S\ref{sec:pss-appendix}.
\end{remark}

\subsection{Action and the Floer complex}\label{sec:action_floer_comp}

To a non-degenerate Hamiltonian function $H: M\times S^{1} \rightarrow \mathbb{R}$ one associates an action functional:
\begin{equation*}
        \mathcal{A}_{H}(\overline{x}) = \int_0^1H(t,x(t))dt +\int_{\overline{x}}\omega,
\end{equation*}
on the aforementioned covering of $\mathcal{L} M$, whose critical points $\mathcal{P}(H)$ are exactly the cappings $\overline{x}$ of contractible loops $x(t)$ solving $x'(t)=X_t(x(t))$.

For a generic $\omega$-tame almost complex structure $J$, the Floer complex is composed of the $\mathbf{F}$-vector space $\mathrm{CF}(H)$ of finite sums\footnote{Warning: the use of finite sums here is specific to the setting of a monotone symplectic manifold $M$ --- in general one requires working with a completion of this vector space.} generated by capped orbits; the differential is defined as the signed\footnote{We refer to \cite{floer-hofer-math-z-1993} for details on the signs appearing in Hamiltonian Floer theory.} count of finite-energy rigid-up-to-translations solutions \emph{Floer's equation}:
\begin{equation*}
  \bd_{s}u+J(u)(\bd_{t}u-X_{t}(u))=0;
\end{equation*}
the input is the positive end $u(\infty,t)$; this is what we mean by ``cohomological conventions''.
The complex is graded in such a way that $\mathrm{PSS}(C)$ has degree equal to the codimension of $C$; see \S\ref{sec:pss-appendix}. The usual index formula \cite{schwarz-thesis} shows that the degree is equal to:\footnote{We use the conventions in \cite{cant-thesis-2022} for the Conley-Zehnder indices (these are fairly common conventions). They depend on a choice of section $\mathfrak{s}$ of the anticanonical bundle $\mathrm{det}_{\C}(TM)$ which is non-vanishing on the image of $x$; this section also determines the number $c_{1}(\bar{x})$ as a homological intersection number of $\bar{x}$ with the cooriented cycle $\mathfrak{s}^{-1}(0)$. The resulting degree is independent of the choice of $\mathfrak{s}$.}
\begin{equation*}
  \mathrm{degree}(\bar{x})=n+2c_{1}(\bar{x})-\mathrm{CZ}(x);
\end{equation*}
With the aformentioned cohomological conventions the differential strictly increases action and has degree one.

The homology of the Floer complex is denoted by $\mathrm{HF}^{*}(H,\mathbf{F})$.  It does not depend on the auxiliary choices. It carries a module structure over $\mathbf{F}[q^{-1},q]$ where the action of $q$ is given by $q\cdot\bar{x} = \bar{x}\# A_{0}$, where $c_{1}(A_{0}) = N$ is the minimal Chern number (i.e., $q$ acts by ``recapping'').

\subsection{$\mathrm{PSS}$-isomorphism}\label{sec:pss-appendix}
There is a ring isomorphism:
\begin{equation*}
  \mathrm{PSS}:\mathrm{QH}^*(M;\mathbf{F})\rightarrow \mathrm{HF}^*(H)
\end{equation*}
which intertwines the quantum product with the \emph{pair-of-pants product} on the Hamiltonian Floer cohomology (we do not discuss the pair of pants product in this paper, and refer the reader to \cite{schwarz-thesis,piunikhin-salamon-schwarz-1996}). Here $\mathrm{QH}^{*}(M;\mathbf{F})$ is as in \eqref{eq:quantum-cohomology}. This map is called the Piunikhin–Salamon–Schwarz isomorphism. With our conventions, $\mathrm{PSS}(C)$ is defined as the sum of solutions $u$ of a certain asymptotic boundary value problem (illustrated in Figure \ref{figure:cappings-PSS}) with a constraint $u(\infty)\in C$; this sum is considered as a sum of cohomologically capped orbits.

Let us be precise in the case of $M=\mathrm{Gr}(k,n)$: recall that $H^{*}(M;\mathbf{F})$ is the free $\mathbf{F}$-vector space generated by the symbols $\sigma_{D}$ where $D$ is a Young diagram fitting into a rectangle of dimensions $n-k$ by $k$. Each Young diagram $D$ determines a pseudocycle $C_{D}$, whose codimension is (twice) the number of boxes in $D$; see Figure \ref{fig:young-diagram} for an illustration. One defines $\mathrm{PSS}(\sigma_{D})$ by counting the solutions of the PSS-equation passing with constraint on $C_{D}$. One extends the definition of $\mathrm{PSS}$ to $\mathrm{QH}^{*}(M;\mathbf{F})$ using the aforementioned $\mathbf{F}[q^{-1},q]$ module structure on $\mathrm{HF}^{*}(F)$; this is a graded map.

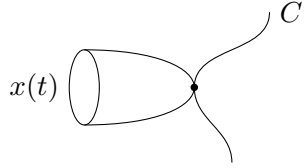
\begin{figure}[h]
  \centering
  \begin{tikzpicture}
    \draw (0,0) circle (0.2 and 0.5) +(-0.2,0) node[left]{$x(t)$};
    \draw (0,0.5) to[out=0,in=0,looseness=5]coordinate[pos=0.5](A) (0,-0.5);
    \draw (A) +(1,1)node[right]{$C$}to[out=-90,in=90]+(0,0)node[fill,circle,inner sep=1pt]{}to[out=-90,in=90]+(.5,-1);
  \end{tikzpicture}
  \caption{A solution to the PSS equation is a map $u:\mathrm{CP}^{1}-\set{0}\to W$ as in \cite{piunikhin-salamon-schwarz-1996}. The figure illustrates one solution of the PSS equation defining the Floer cycle $\mathrm{PSS}(C)$; this cycle is the sum over all such rigid solutions considered as cappings (note that the loop $x(t)$ is a negative boundary of the solution).}
  \label{figure:cappings-PSS}
\end{figure}

\subsection{Spectral invariants and the spectral norm}\label{sec:spec_invs}

To a quantum cohomology class $e\in \mathrm{QH}^{*}(M;\mathbf{F})$ and a Hamiltonian $H_{t}$, one associates a real number, called a spectral invariant, defined by a cohomological max-min:
$$c(e;H_{t}):= \sup\set{\min\set{\mathcal{A}_{H}(\bar{x}_{1}),\mathcal{A}_{H}(\bar{x}_{2}),\dots}:[\sum \bar{x}_{i}]=[\mathrm{PSS}(e)]};$$
one maximizes over all elements in a given cohomology class. It is well known that spectral invariants take values in the set of critical values of the action functional $\mathcal{A}_H$.

The \emph{spectral norm} of a Hamiltonian function $H$ is defined by: $$\gamma(H)=-c(1;H)-c(1;\bar{H});$$ it depends only on the class in the universal cover of $\mathrm{Ham}(M)$ represented by the Hamiltonian isotopy $\phi_H^t$ generated by $H$. Here, $\bar{H}(t,x)=-H(1-t,x)$.

\bibliographystyle{./amsalpha-doi.bst}
\bibliography{citations}
\clearpage

\end{document}